\documentclass[11pt]{article}

\usepackage{fncylab,enumerate,wrapfig,placeins}
\usepackage{bbm}
\usepackage{graphicx}
\usepackage{amsmath,amssymb}  
\usepackage{amsmath,amsthm,amsfonts,amssymb}  
\usepackage{eufrak}
 
\usepackage{accents}

\usepackage{geometry}
\pdfoutput=1  %  for arxiv

\usepackage{subcaption}
\usepackage{booktabs}

\usepackage{multirow}

\usepackage{algorithm}
\usepackage{algpseudocode}

\usepackage{algorithm}
\usepackage{algpseudocode}
\algnewcommand\algorithmicinput{\textbf{Input:}}
\algnewcommand\INPUT{\item[\algorithmicinput]}
\algnewcommand\algorithmicoutput{\textbf{Output:}}
\algnewcommand\OUTPUT{\item[\algorithmicoutput]}

\usepackage[usenames,dvipsnames]{color}
\usepackage{color}

\usepackage[colorlinks,%
linkcolor=BrickRed,% 
filecolor =BrickRed,%	
citecolor=BrickRed,%
]{hyperref}

\graphicspath{{SFQFigures/}}

\newlength{\noteWidth}
\long\def\notes#1{\ifinner
{\footnotesize #1}
\else 
\marginpar{\parbox[t]{\noteWidth}{\raggedright\footnotesize#1}}
\fi\typeout{#1}}

\def\notes#1{}%\typeout{read notes: #1}}  %uncomment for final version

\def\spm#1{\notes{SPM:  #1}} 

\def\sh#1{\notes{SH: #1}}

\def\sh#1{\notes{SH:  #1}}

\def\bl#1{{\color{blue}#1}} 
\newtheorem{theorem}{Theorem}[section]
\newtheorem{corollary}[theorem]{Corollary}
\newtheorem{proposition}[theorem]{Proposition}
\newtheorem{lemma}[theorem]{Lemma}

\def\Lemma#1{Lemma~\ref{#1}}
\def\Proposition#1{Prop.~\ref{#1}}
\def\Theorem#1{Theorem~\ref{#1}}
\def\Corollary#1{Corollary~\ref{#1}}

\def\Section#1{Section~\ref{#1}}

\newcounter{rmnum}
\newenvironment{romannum}{\begin{list}{{\upshape (\roman{rmnum})}}{\usecounter{rmnum}
\setlength{\leftmargin}{14pt}
\setlength{\rightmargin}{8pt}
\setlength{\itemsep}{2pt}
\setlength{\itemindent}{-1pt}
}}{\end{list}}

\newcounter{anum}

 \def\FRAC#1#2#3{\genfrac{}{}{}{#1}{#2}{#3}}

\def\ddt{{\mathchoice{\FRAC{1}{d}{dt}}%
{\FRAC{1}{d}{dt}}%
{\FRAC{3}{d}{dt}}%
{\FRAC{3}{d}{dt}}}}

\def\ddtp{{\mathchoice{\FRAC{1}{d^{\hbox to 2pt{\rm\tiny +\hss}}}{dt}}%
{\FRAC{1}{d^{\hbox to 2pt{\rm\tiny +\hss}}}{dt}}%
{\FRAC{3}{d^{\hbox to 2pt{\rm\tiny +\hss}}}{dt}}%
{\FRAC{3}{d^{\hbox to 2pt{\rm\tiny +\hss}}}{dt}}}}

\def\half{{\mathchoice{\FRAC{1}{1}{2}}%
{\FRAC{1}{1}{2}}%
{\FRAC{3}{1}{2}}%
{\FRAC{3}{1}{2}}}}

\newcommand{\field}[1]{\mathbb{#1}}

\def\Re{\field{R}}

\def\ind{\field{I}}

\def\transpose{{\intercal}}

\def\Prob{{\sf P}}

\def\Expect{{\sf E}}

\def\epsy{\varepsilon}

\def\varble{\,\cdot\,}

\def\state{{\sf X}}

\def\zstate{{\sf Z}}

\def\haf{{\hat f}}
\def\hag{{\hat g}}

\def\bfmath#1{{\mathchoice{\mbox{\boldmath$#1$}}%
{\mbox{\boldmath$#1$}}%
{\mbox{\boldmath$\scriptstyle#1$}}%
{\mbox{\boldmath$\scriptscriptstyle#1$}}}}

\def\bfmX{\bfmath{X}}

\def\bfPhi{\bfmath{\Phi}}

\makeatletter
\newcommand{\dotDelta}{{\vphantom{\Delta}\mathpalette\d@tD@lta\relax}}
\newcommand{\d@tD@lta}[2]{%
\ooalign{\hidewidth$\m@th#1\mkern-1mu\cdot$\hidewidth\cr$\m@th#1\Delta$\cr}%
}
\makeatother

\def\Alg#1{Alg.~\ref{a:Zapalg}}

\def\fSA{f}
\def\barfSA{\bar{f}}

\def\eqdef{\mathbin{:=}}

\def\elig{\zeta}

\def\trace{\hbox{\rm trace\,}}

\def\Lemma#1{Lemma~\ref{#1}}
\def\Proposition#1{Prop.~\ref{#1}}
\def\Prop#1{Prop.~\ref{#1}}
\def\Theorem#1{Thm.~\ref{#1}}   %Removed Theorem to be consistent with Prop.
\def\Corollary#1{Corollary~\ref{#1}}

\def\hab{\widehat b}
\def\haA{\widehat A}

\def\Var{\hbox{\sf Var}\,}

\def\Cov{\text{\rm Cov}\,}

\graphicspath{{SFQFigures/}}

\def\ind{\field{I}}

\def\Re{\field{R}}

\def\LV{L_\infty^V}
\def\LVstar{L_\infty^{V^*}}

\newlength{\dhatheight}
\newcommand{\doublehat}[1]{%
    \settoheight{\dhatheight}{\ensuremath{\hat{#1}}}%
    \addtolength{\dhatheight}{-0.25ex}%
    \hat{\vphantom{\rule{1pt}{\dhatheight}}%
    \smash{\hat{#1}}}}

\def\hahaf{\doublehat{f}}

\def\clA{\mathcal{A}}
\def\clB{\mathcal{B}}
\def\clD{\mathcal{D}}
\def\clE{\mathcal{E}}
\def\clF{\mathcal{F}}
\def\clM{\mathcal{M}}
\def\clT{\mathcal{T}}
\def\clX{\mathcal{X}}

\def\tot{{\text{tot}}}

\def\Amap{\clA}

\def\Bmap{\clB}

\def\haclA{\widehat{\clA}}

\def\tiltheta{\widetilde \theta}

\def\haZ{\widehat Z}

\def\haDelta{\widehat \Delta}

\def\sqrtV{\scriptscriptstyle\sqrt{V}}

\title{Explicit Mean-Square Error Bounds
	\\
for Monte-Carlo and Linear Stochastic Approximation} 

\author{
		Shuhang Chen\thanks{Department of Mathematics at the University of Florida, Gainesville.}%
		\and
		Adithya M. Devraj\thanks{Department of ECE at the University of Florida, Gainesville.}%
		\and
		Ana Bu\v{s}i\'{c}\thanks{Inria and DI ENS, \'Ecole Normale	Sup\'erieure, CNRS, PSL Research University, Paris, France.      \newline Financial support from ARO grant W911NF1810334 is gratefully acknowledged.  Additional support from EPCN 1609131 \&\ CPS~1646229, and French National Research Agency grant ANR-16-CE05-0008.}%
		\and 
		Sean  Meyn\footnotemark[2]
}

\begin{document}

\maketitle

%The Abstract paragraph should be indented 0.25 inch (1.5 picas) on both left and right-hand margins. Use  10~point type, with a vertical
%spacing of 11~points.  {\bf Abstract} must be centered, bold, and in  point size 12. Two line spaces precede the Abstract. The Abstract
%must be limited to one paragraph.

\begin{abstract}
%Stochastic approximation (SA) algorithms find wide variety of applications in statistics and machine learning. 

This paper concerns error bounds for recursive equations subject to Markovian disturbances.   Motivating examples abound within the fields of Markov chain Monte Carlo (MCMC) and Reinforcement Learning (RL),   and many of these algorithms can be interpreted as special cases of stochastic approximation (SA).   It is argued that it is not possible in general to obtain a Hoeffding bound on the error sequence, even when the underlying Markov chain is reversible and geometrically ergodic, such as the M/M/1 queue.  This is motivation for the focus on mean square error bounds for parameter estimates.  It is shown that  mean square error achieves the optimal rate of $O(1/n)$, subject to conditions on the step-size sequence.    Moreover, the exact constants in the rate are obtained,   which is of great value in algorithm design.   
% We demonstrate that our results are applicable to MCMC and temporal difference learning.
\medskip

\smallskip

%{\small
%	\noindent
%	\textbf{2000 AMS Subject Classification:}
%	93E20,	%  	Optimal stochastic control
%	93E35	%  	Stochastic learning and adaptive control
%	%60J20  	%Applications of Markov chains and discrete-time Markov processes on general state spaces (social mobility, learning theory, industrial processes, etc.) [See also 90B30, 91D10, 91D35, 91E40]
%	%60J22  	%Computational methods in Markov chains [See also 65C40]
%	
%	% 60J10,          %  chains with discrete parameter
%	%60J25,          % Markov processes with continuous parameter
%	%37A30,          % Ergodic theorems, spectral theory, Markov operators
%	% 60F10,          % Large deviations
%	%47H99.          % nonlinear operators
%}
\end{abstract}

%\textbf{NOTATION (temporary)}
%
%I want consistency with Adithya's thesis.  So,
%\[
%\Sigma_\theta = \lim_{n\to\infty}  \Sigma_n
%\]
%
%
%$\tiltheta^{\varrho}_n =  n^\varrho \tiltheta_n$ 
%
%$\Sigma_n^{\varrho,(i)} =\Expect[\tiltheta^{\varrho,(i)}(\tiltheta^{\varrho,(i)})^\transpose] = n^{2\varrho}\Cov(\theta_n^{(i)})$.
%(where is this defined in the paper?  Maybe not used)
%Special case $\varrho =1$ we omit the variable:  $\Sigma_n^{(i)} =  n\Cov(\theta_n^{(i)})$.
%
%
%
%\[
%\Cov(\theta_n) = \sum_{i=1}^3 \Cov(\theta_n^{(i)}) + \sum_{i=1}^3\sum_{j=1, j\neq i}^3 \Expect[\tiltheta_n^{(i)} (\tiltheta_n^{(j)})^\transpose]
%\]
%
%
%Also,   please maintain my notation.  It was disturbing to see \verb+\clE+ suddenly being used for $\tiltheta$.    
% 

\section{Introduction}

%\rd{Need to adjust list  (see macro in source file)}
%
%\renewenvironment{romannum}{\begin{list}{{\upshape (\roman{rmnum})}}{\usecounter{rmnum}
%\setlength{\leftmargin}{3pt}
%\setlength{\rightmargin}{8pt}
%\setlength{\itemsep}{1pt}
%\setlength{\itemindent}{16pt}
%\setlength{\topsep}{0pt}
%\setlength{\partopsep}{0pt} 
%}}{\end{list}}

Many questions in statistics and the area of reinforcement learning are concerned with computation of the root of a function in the form of an expectation:  $\barfSA(\theta)=\Expect[\fSA(\theta,\Phi))]$,  where $\Phi$ is a vector valued random variable, and $\theta\in\Re^d$.   The value $\theta^*$ satisfying $\barfSA(\theta^*)=0$ is most commonly approximated through some version of the stochastic approximation (SA) algorithm of Robbins and Monro \cite{robmon51a,bor08a}.   In its basic form, this is the recursive algorithm

\begin{equation}
\theta_{n+1} = \theta_n + \alpha_{n+1} \fSA(\theta_n, \Phi_{n+1})
\label{e:SA}
\end{equation} 
in which $\{\alpha_n\}$ is a non-negative gain sequence,   and $\{ \Phi_n\}$ is a sequence of random variables whose distribution converges to that of $\Phi$ as $n\to\infty$.   
The sequence is a   Markov chain in the applications of interest in this paper.

There is a large body of work on conditions for convergence of this recursion, and also a Central Limit Theorem (CLT):  with $ \tiltheta_n  =\theta_n-\theta^*$,
\begin{alignat*}{2}
\lim_{n\to\infty} \tiltheta_n  &= 0     &&\textit{almost surely}
\\
\lim_{n\to\infty}  \sqrt{n} \tiltheta_n              & =   N(0,\Sigma_\theta)  \qquad    &&\textit{in distribution}
\end{alignat*} 
The $d\times d$ matrix $\Sigma_\theta$ is known as the \textit{asymptotic covariance}.  
The CLT requires substantially stronger assumptions on the gain sequence, the function $\fSA$,  
and the statistics of the ``noise'' sequence $\{ \Phi_n\}$   \cite{benmetpri12,kusyin97}.

%{\huge  \rd{
%AD:  revise history to include
%\cite{blu54}
%}}

Soon after the stochastic approximation algorithm was first introduced in \cite{robmon51a,blu54},  
Chung~\cite{chu54o}  identified the optimal CLT covariance and techniques to obtain the optimum for scalar recursions.
This can be cast as a form of \textit{stochastic Newton-Raphson} (SNR) \cite{devmey17a,devmey17b,devbusmey19,dev19}.
Gradient free methods  [or  \textit{stochastic quasi Newton-Raphson} (SQNR)] appeared in later work:   The first example was proposed by 
Venter in \cite{ven67}, which was shown to obtain the optimal variance for a one-dimensional SA recursion.  The algorithm obtains estimates of the SNR gain $-A^{-1}$ (see   \eqref{e:linSA}  below), through  a procedure similar to the Kiefer-Wolfowitz algorithm \cite{kiewol52}.      Ruppert proposed an extension of Venter's algorithm for vector-valued functions \cite{rup85}. 

The averaging technique of Ruppert and Polyak \cite{rup88,pol90,poljud92} is a two-time-scale algorithm that is also designed to achieve the optimal asymptotic covariance. More recently, a two-time-scale variant of the SNR algorithm known as ``Zap-SNR" was proposed in \cite{devmey17a,devmey17b,devbusmey19,dev19}, with applications to reinforcement learning.   Zap~algorithms are stable and convergent under mild assumptions \cite{devmey17a, chedevbusmey19}.

Under the typical assumptions under which the CLT holds for the recursion \eqref{e:SA},  the asymptotic covariance  has an explicit form in terms of a linearization \cite[Chapter 10, Theorem 3.3]{kusyin03}.  Assume that the solution to $\barfSA(\theta^*)=0$ is unique, and denote  
\begin{equation}
A = \partial  \barfSA\, (\theta^*)   \,,\qquad  \Delta_n  =  \fSA(\theta^*, \Phi_n)
\label{e:linSA}
\end{equation}
The error dynamics of the SA recursion are then approximated by the linear SA recursion: 
\begin{equation}
\tiltheta_{n+1}  = \tiltheta_n + \alpha_{n+1} [A \tiltheta_n  + \Delta_{n+1}  ]  \,.   
\label{e:SAlinearized}
\end{equation}
Subject to the assumption that $\half I + A$ is Hurwitz (i.e., $\text{Real}(\lambda) < -\half$ for each eigenvalue  of $A$), 
the $d\times d$ matrix $\Sigma_\theta$ is the unique positive semi-definite solution to the Lyapunov equation
\begin{equation}
[\half I+A]\Sigma + \Sigma[ \half I + A]^\transpose + \Sigma_{\Delta}=0
\label{e:lyap}
\end{equation}
in which $\Sigma_{\Delta}$ is also an asymptotic covariance:  the covariance matrix appearing in the CLT for the sequence $\{\Delta_n\}$   (which may be expressed in terms of a solution to a Poisson equation - see \cite[Theorem 2.2, Chapter 10]{kusyin03}).

The goal of this paper is to demonstrate that the CLT is far less \textit{asymptotic} than it may appear. For this we focus analysis on the linearization \eqref{e:SAlinearized},  along with first steps towards analysis of the nonlinear recursions.   
Subject to assumptions on $A$ and the Markov chain, we establish the bound 
\begin{equation}
\Cov(\theta_n)  =  n^{-1}  \Sigma_\theta  + O(n^{-1-\delta})  
\label{e:finite-n1}
\end{equation}
with $\Cov(\theta_n) = \Expect[\tiltheta_n\tiltheta_n^\transpose]$.   
Under further assumptions, the bound is refined to obtain $\delta\ge 1$,  and even finer bounds:
\begin{equation}
\Cov(\theta_n)  =  n^{-1}  \Sigma_\theta  +n^{-2}\Sigma_{\theta,2}   + O(n^{-2-\delta})  
\label{e:finite-n2}
\end{equation}
where again $\delta>0$ and formula of $ \Sigma_{\theta,2}$ is obtained in the paper based on a second Lyapunov equation and a  solution to a \textit{second} Poisson equation. 

It is hoped that these results will be helpful in construction and performance analysis of many algorithms found in machine learning, statistics and reinforcement learning.   Identification of the coefficient for the $n^{-2}$ term from \eqref{e:finite-n2} may lead to criteria for gain design when one aims to minimize the covariance with a fixed budget on the number of iterations.

The reader may ask, why not search directly for finite-$n$ bounds of the flavor of Hoeffding's inequality:
\begin{equation}
\Prob\{ \|\tiltheta_n \| \ge \epsy \}   \le  b_0 \exp( - n I_0(\epsy) )\, 
\label{e:Hoeff}
\end{equation}
where $b_0>0$ is fixed,  and $I_0$ is a convex function that is strictly positive and finite in a region $0<\epsy^2 \le \bar\epsy^2$.   The answer is that such bounds are not always possible even for the simplest SA recursions, even when the Markov chain is geometrically ergodic.    This is clarified in the first general example:

\subsection{Markov Chain Monte Carlo}
\label{s:MCMC}

As a prototypical example of stochastic approximation, Markov chain Monte Carlo (MCMC) proceeds by constructing an ergodic Markov chain $\bfPhi$ with invariant measure $\pi$ so as to estimate $\pi(F) = \int F(z)\, \pi(dz)$ for some function $F:\zstate\rightarrow \Re^d$. One then simulates $\Phi_1,\Phi_2,\dots,\Phi_{n+1}$ to obtain the estimates
\begin{equation}
\label{e:mcmc-est}
\theta_{n+1} = \frac{1}{n+1}\sum_{k=1}^{n+1} F(\Phi_k)
\end{equation}
This is an instance of the SA recursion \eqref{e:SA}:
\begin{equation}
\label{e:mcmc-sa}
\theta_{n+1} = \theta_n + \frac{1}{n+1}(- \theta_n + F(\Phi_{n+1}))
\end{equation}
Subtracting $\theta^* = \pi(F)$ from both sides of \eqref{e:mcmc-sa} gives, with $\tiltheta_n = \theta_n - \pi(F)$,
\[
\tiltheta_{n+1}  = \tiltheta_n  + \frac{1}{n+1}(- \tiltheta_n  + F(\Phi_{n+1}) - \pi(F) )
\]
which is   \eqref{e:SAlinearized} in a special case: $A=-I$,  $\Delta_{n+1} = F(\Phi_{n+1}) - \pi(F)$ and $\alpha_n=1/n$.

A significant part of the literature on MCMC focuses on finding Markov chains whose marginals approach the invariant measure $\pi$ quickly. Error estimates for MCMC   have only been studied under rather restrictive settings. For instance, under the assumption of uniform ergodicity of $\bfPhi$  and uniform boundedness of $F$  (which rarely hold in practice outside of a finite state space),   a generalized Hoeffding's inequality was obtained in \cite{glyorm02} to obtain the PAC-style error bound \eqref{e:Hoeff}.   
%\sh{Should refer some paper on MCMC convergence rate here? like \cite{ros95a}?
%\\
%SM2SH:   no, I believe this is just about rates for ergodicity}
We can not expect Hoeffding's bound if either of these assumptions is relaxed.   Consider the simplest countable state space Markov chain:  the M/M/1 queue with  uniformization,  defined with $\zstate=\{0,1,2,\dots\}$ and
\[
\Phi_{n+1} = \begin{cases}    \Phi_n+1   &   \text{prob.\ $\alpha$}
\\
\max(\Phi_n-1, 0)  &   \text{prob.\ $\mu = 1-\alpha$}
\end{cases}
\]
This is a reversible, geometrically ergodic Markov chain when $\rho=\alpha/\mu<1$, with geometric invariant distribution $\pi(z) = (1-\rho)\rho^z$,  $z\ge 0$.  
It is shown in \cite{CTCN} that the error bound \eqref{e:Hoeff} fails for most unbounded functions $F$.    The question is looked at in greater depth in~\cite{dufmey10,dufmey14}, where asymptotic bounds are obtained for the special case  $F(z) \equiv z$.    An asymptotic version of  \eqref{e:Hoeff} is obtained for the lower tail:
\begin{align}
\lim_{n\to\infty}   \frac{1}{n} \log \Bigl(
\Prob\{  \tiltheta_n   \le -\epsy \}  \Bigr) &=  - I_0(\epsy)  
\label{e:HoeffMM1-}
\end{align}
in which the right hand side is strictly negative and finite valued for positive $\epsy$ in a neighborhood of zero.  
An entirely different scaling is required for the upper tail:
\begin{align}
\lim_{n\to\infty}  \frac{1}{n}  \log \Bigl(
\Prob\Bigl \{    \frac{\tiltheta_n}{n}  \ge \epsy   \Bigr\}  \Bigr)  &=  - J_0(\epsy)  
\label{e:HoeffMM1+}
\end{align}
where again the right hand side is  strictly negative and finite valued for   $\epsy>0$ sufficiently small.   It follows from \eqref{e:HoeffMM1+} that the PAC-style bound  \eqref{e:Hoeff}  is not attainable.

\subsection{Reinforcement Learning}
\label{s:rl-td}
%The theory of this paper is illustrated here through application to TD-learning. 

The theory of this paper also applies to TD-learning.   In this case, the Markov chain $\bfPhi$ contains as one component a state process for a system to be controlled.

Consider a Markov chain $\bfmX$ evolving on a (Polish) state space $\state$.
Given a cost function $c: \state \to \Re$, and a discount factor $\beta \in (0,1)$, the goal in TD-learning is to approximate the solution $h: \state \to \Re$ to the Bellman equation:
\begin{equation}
h(x) = c(x) + \beta \Expect [ h(X_{n+1}) \mid X_n = x  ]
\label{e:BE}
\end{equation}

This functional equation can be recast:   
\begin{subequations}
	\begin{alignat}{1}
\Expect[&\clD(h, \Phi_{n+1}) \mid \Phi_0\,\dots\, \Phi_n] = 0
\label{e:BE_fun} 
\\[.2em]
\textit{where} \quad
&
\clD(h, \Phi_{n+1})    \eqdef c(X_n) + \beta h(X_{n+1}) - h(X_n)\,,  \quad \Phi_{n+1} \eqdef (X_{n+1} , X_{n}) \,,  \quad n\ge 0
\label{e:td_F}
\end{alignat}
\end{subequations} 
Equation \eqref{e:BE_fun} may be regarded as motivation for the TD-learning algorithms of  \cite{sut88,tsiroy97a}.

Consider a linearly parameterized family of candidate approximations 
$\{h^\theta(x) = \theta^\transpose \psi (x) : \theta\in\Re^d\}$, where  $\psi: \state \to \Re^d$ denotes the $d$ basis functions. The goal in TD-learning is to solve the \emph{Galerkin relaxation} of (\ref{e:BE_fun},\ref{e:td_F}):  % Find $\theta^* \in \Re^d$ such that
\begin{equation}
\Expect[\clD(h^{\theta^*}, \Phi_{n+1}) \elig_n ] = 0
\label{e:BE_fun_gal}
\end{equation} 
where $\{\elig_n\}$ is a $d$-dimensional stochastic process, adapted to     $\bfPhi$, and the expectation is with respect to the steady state distribution.
In particular, $\elig_n \equiv \psi (X_n)$ in  TD($0$) learning, so that the goal in this case is to find $\theta^* \in \Re^d$ such that: 
\begin{equation}
\begin{aligned}
\Expect &\big [  \clD(h^{\theta^*}, \Phi_{n+1}) \psi(X_n) \big ] =  0
\\
& \clD(h^{\theta^*}, \Phi_{n+1})   =  c(X_n)  + \beta h^{\theta^*} (X_{n+1}) -h^{\theta^*} (X_n) 
\end{aligned}
\label{e:td0_F}
\end{equation}

The TD($0$) algorithm is the SA recursion \eqref{e:SA} applied to solve \eqref{e:td0_F}:
\begin{equation}
\begin{aligned}
\theta_{n+1} & = \theta_n + \alpha_{n+1} d_{n+1} \psi(X_n)
\\
d_{n+1} & =  c(X_n)  +  \beta h^{\theta_n}(X_{n+1}) - h^{\theta_n}(X_n)  
\label{e:TD0_SA}
\end{aligned}
\end{equation}
Denoting
\begin{equation*}
\begin{aligned}
A_{n+1} & \eqdef \psi(X_n) \big( \beta \psi(X_{n+1}) - \psi(X_n) \big)^\transpose 
\\
b_{n+1} & \eqdef - c(X_n) \psi(X_n)
\end{aligned}
\end{equation*}
the algorithm \eqref{e:TD0_SA} can be rewritten as:
\begin{equation}
\theta_{n+1} = \theta_n + \alpha_{n+1} \big( A_{n+1} \theta_n  -  b_{n+1} \big)
\label{e:TD0_linSA}
\end{equation}
Note that $\theta^*$ from \eqref{e:BE_fun_gal} solves the linear equation $\Expect [ A_{n+1} ]\theta^* =\Expect[b_{n+1}]$. Subtracting $\theta^*$ from both sides of \eqref{e:TD0_linSA} gives, with $\tiltheta_n = \theta_n - \theta^*$,
\begin{equation}
\tiltheta_{n+1} = \tiltheta_n + \alpha_{n+1}[A\tiltheta_n + (A_{n+1}-A)\tiltheta_n + A_{n+1}\theta^* - b_{n+1}]
\label{e:TD0_linSA_err}
\end{equation}
Under mild conditions, we show through coupling  that iteration \eqref{e:TD0_linSA_err} can be closely approximated by the linear SA recursion \eqref{e:SAlinearized} with matrix $A=\Expect [ A_{n+1} ]$ and noise sequence $\Delta_{n+1} = A_{n+1}\theta^* - b_{n+1}$. In particular, the two recursions have the same asymptotic covariance if the matrix $\half I + A$ is Hurwitz  (see \Section{s:mse}).
 
Under general assumptions on the Markov chain $\bfmX$, and the basis functions $\psi$,  it is known that matrix $A = \partial  \Expect[\clD(h^{\theta^*}, \Phi_{n+1})] = \Expect [ A_{n+1} ]$ is Hurwitz, and that  the sequence of estimates $\{\theta_n\}$ converges to $\theta^*$ \cite{tsiroy97a}.
However, when the discount factor $\beta$ is close to $1$, it can be shown that $\lambda_{\max} > - \half$ (where $\lambda_{\max}$ denotes the largest eigenvalue of $A$), and is in fact close to $0$ under mild additional assumptions \cite{devmey17a,dev19,devbusmey20}.  It follows that the algorithm has infinite asymptotic covariance:  full details and finer results can be deduced from  Theorems~\ref{t:SAlinearized} and \ref{t:couple-td}.

%\ad{Should we state the eigenvector condition? I avoided this because then I'll have to define the noise sequence and then $\Sigma_\Delta$.
%\\
%SM2AD:  agreed}

The SNR algorithm is defined as follows:  
%\spm{fair to exclusively cite us for this? \cite{devmey17a}.  Let's leave out reference since we say so much about SNR being optimal above. }
\begin{align}
\theta_{n+1} & = \theta_n - \alpha_{n+1} \haA_{n+1}^{-1} d_{n+1} \psi(X_n)
\label{e:TD0_SNR}
\\
d_{n+1} & =  c(X_n) + \beta h^{\theta_n}(X_{n+1}) - h^{\theta_n}(X_n)  
\nonumber
\\
\haA_{n+1} & =  \haA_n   +  \alpha_{n+1} \big[ A_{n+1} -  \haA_n  \big]
\label{e:TD0_haA_SNR}
\end{align}
Under the assumption that the sequence of matrices $\{\haA_n: n \geq 0\}$ is invertible for each $n$, it is shown in \cite{devmey17a,dev19} that the sequence of estimates obtained using (\ref{e:TD0_SNR},\ref{e:TD0_haA_SNR}) are identical to the parameter estimates obtained using the LSTD($0$) algorithm:  $\theta_{n+1}   = \haA_{n+1}^{-1} \hab_{n+1}$, with
%\sh{Need to put TD error bounds those RL people might know.}
\begin{align} 
\haA_{n+1} & =  \haA_n   +  \alpha_{n+1} \big[ A_{n+1} -  \haA_n  \big]
\nonumber
\\
\hab_{n+1} & =  \hab_n   +  \alpha_{n+1} \big[ b_{n+1} -  \hab_n  \big]
\nonumber
\end{align}
Consequently, the LSTD($0$) algorithm achieves the optimal asymptotic covariance.

%\bl{\bf Response to Ana:}
Q-learning and many other RL algorithms can also be cast as SA recursions.  They are no longer linear,   but it is anticipated that bounds can be obtain in future research through linearization \cite{ger99}.

\subsection{Literature Survey}

Finite time performance  bounds
for linear stochastic approximation were obtained in many prior papers, subject to the assumption that the noise sequence $\{\Delta_n\}$ appearing in \eqref{e:SAlinearized} is a martingale difference sequence
\cite{dalszothoman2017finite, narsze2017finite}.    This assumption is rarely satisfied in the applications of interest to the authors.  

Much of the literature on finite time bounds for linear SA recursions with Markovian noise has been recent. For constant step-size algorithms with step-size $\alpha$,  it follows from analysis in \cite{bormey00a} that the pair process $(\theta_n , \Phi_n)$ is a geometrically ergodic Markov chain, and the covariance of $\theta_n$ is $O(\alpha)$ in steady state.
Finite time bounds of order $O(\alpha)$ were obtained in \cite{tad06, bharussin2018finite, sriyin19, hu19}.
Unfortunately,  these bounds are not tight, and hence their value for algorithm design is limited. 

Mean-square error bounds have also been obtained for diminishing step-size algorithms,   to establish the optimal rate of convergence 
$\Expect[\|\tiltheta_n\|^2] \le b_\theta  / n$   \cite{sriyin19, bharussin2018finite, che19}.     The constant $b_\theta  $  is a function of the mixing time of the underlying Markov chain.
These results require strong assumptions (uniform ergodicity of the Markov chain),  and do not obtain the optimal constant $b^*_\theta =\trace(\Sigma_\theta)$.     Rather than parameter estimation error,  finite time bounds are obtained in \cite{karmiamouwai19} for $\Expect[\|\barfSA(\theta_n)\|^2] $,  which may be regarded as a far more relevant performance criterion.   Bounds are obtained for Markovian models,  subject to the existence of a Lyapunov function similar to what is assumed in the present work.  It is again not clear if the resulting bounds are tight,  or have value in algorithm design.

\subsection{Contributions}

The main contribution of this paper is a general framework for analyzing the finite time performance of linear stochastic approximation algorithms with Markovian noise,
and vanishing step-size (required to achieve the optimal rate of convergence of Chung-Ruppert-Polyak).
The M/M/1 queue example illustrates plainly that Markovian noise introduces challenges not seen in the ``white noise'' setting,  and that the finite-$n$   error bound \eqref{e:Hoeff} cannot be obtained without substantial restrictions.   Even under the assumptions of \cite{glyorm02}  (uniform ergodicity, and bounded noise),  the resulting bounds are \textit{extremely loose} and hence may give little insight for algorithm design.  
Our approach allows us to obtain explicit bounds,  without the uniform boundedness assumption of noise that is frequently imposed in the literature~\cite{bharussin2018finite, sriyin19, hu19, che19}.   Instead,  it is assumed that the Markovian noise is  $V$-uniform ergodic;  an assumption that is far   weaker than   geometric or uniform mixing.

Our starting point is the classical martingale approximation of the noise used in CLT analysis of Markov chains \cite[Chapter 17]{MT} and used in the analysis of SA recursions since  Metivier and Priouret~\cite{metpri84}.  Under mild assumptions on the Markov chain,  each $\Delta_n$ can be expressed as the sum of a martinagle difference and a telescoping term.   
The solution of the linear recursion \eqref{e:SAlinearized} is decomposed as a sum of the respective responses: 
\begin{equation}
\label{e:SAlinearized-decomp}
\tiltheta_n = \tiltheta_n^{\clM} + \tiltheta_n^{\clT}
\end{equation}
The challenge is to obtain explicit bounds on the mean square error for each term.

We say that a deterministic vector-valued sequence $\{e_n\}$ converges to zero at rate $1/n^{\varrho_0}$ if
\[
\lim_{n\to\infty  } n^{\varrho} \|e_n\|
=
\begin{cases}
0, & \text{if } \;\varrho < \varrho_0 \\
\infty, & \text{if }\; \varrho > \varrho_0
\end{cases}
\]
Bounds for the mean-square error  are obtained in  \Theorem{t:SAlinearized}, subject to conditions on both the matrix $A$ and the noise sequence.   In summary, under general assumptions on $\{\Delta_n\}$, 
\begin{romannum}
\item  The bound 	 \eqref{e:finite-n1} holds if 
$\half I +A$ is Hurwitz.

\item  If $  I +A$ is Hurwitz,   then the finer bound \eqref{e:finite-n2}  holds.

\item If there is an eigenvalue of $A$ satisfying $\text{Real}(\lambda) > -\half$,  and corresponding left-eigenvalue $v$ that lies outside of the null-space of the asymptotic covariance of the noise sequence,  
then
\begin{equation}
\lim_{n\to\infty } n^{2\rho}   \Expect[ |v^\transpose \tiltheta_n|^2]  =
\begin{cases} 
  0,   &  \quad  \varrho < \varrho_0 
\\
 \infty,   & \quad \varrho > \varrho_0 
\end{cases}
\label{e:bad_poly_rate}
\end{equation}
with $\rho_0 = |\text{Real}(\lambda)|$.   The
convergence of $\Expect[ \|\tiltheta_n \|^2]$ to zero is thus no faster than $n^{-2\rho_0}$.
% (every eigenvalue of $A$ has real part less than $-1$), 
%\begin{equation*}
%%	\label{e:finite-n3}
%\Expect[\tiltheta_n\tiltheta_n^\transpose] = n^{-1}\Sigma_\theta +  n^{-2}\bigl(\Sigma_{\theta,2} +\Sigma_Z \bigr) + O(n^{-2-\delta}) 
%\end{equation*}%
%where $\delta >0$, and $\Sigma_{\theta,2}, \Sigma_Z$ are defined in \Theorem{t:SAlinearized-finer}.
\end{romannum}

\section{Mean Square Convergence}

\subsection{Notation and Background}
\label{s:notation}
 
Consider the linear SA recursion \eqref{e:SAlinearized}, with the noise sequence $\{\Delta_n\}$ defined in \eqref{e:linSA}.  We use the following notation to explicitly represent the noise as a function of $\Phi_n$:
\begin{equation}
\fSA^*(\Phi_n) \eqdef \Delta_n = \fSA(\theta^* \,, \Phi_n)
\label{e:f}
\end{equation}
A form of geometric ergodicity is assumed throughout.  To apply standard theory, it is assumed that the state space $\zstate$ is \textit{Polish} (the standing assumption in \cite{MT}).   We fix a measurable function $V\colon\zstate\to [1,\infty)$, and let $\LV$ denote the set of measurable functions $g\colon\zstate\to\Re$ satisfying
\[
\| g\|_V\eqdef \sup_{z\in\zstate}  \frac{|g(z)|}{V(z)}   <\infty
\]
The Markov chain $\bfPhi$ is assumed to be \textit{$V$-uniformly ergodic}: there exists
  $\rho\in(0,1)$, and $B_V<\infty$ such that for each $g\in\LV$, $z \in \zstate$,  
\begin{equation}
\Big| \Expect[g(\Phi_n) \mid \Phi_0 = z]   - \pi(g)  \Big|  
 \le B_V \|g\|_V\rho^n  V(z) \,, \quad n\ge 0
\label{e:Vuni}
\end{equation}
where $\pi$ is the unique invariant measure,   and $\pi(g) = \int g(z) \, \pi(dz)$ is the steady state mean.

The uniform bound \eqref{e:Vuni} is not a strong assumption.  For example, it is satisfied for the  M/M/1 queue described in \Section{s:MCMC} with $V(z) = \exp(\epsy_0 z)$,   for $\epsy_0>0$ sufficiently small,  with $z\in\zstate =\{0,1,\dots\}$~\cite[Thm.~16.4.1]{MT}.

\def\tilg{\tilde g}

The following are imposed throughout:
\paragraph{Assumptions:}
\begin{itemize}
\item[{\textbf{(A1)}}] The Markov process $\bfPhi$ is   $V$-uniformly ergodic, with unique invariant measure denoted $\pi$.
	
\item[{\textbf{(A2)}}] The $d\times d$ matrix $A$ is Hurwitz, and the step-size sequence $\alpha_n \equiv 1/n$, $n\ge 1$. 
% $\Delta_{n+1} = \fSA^*(\Phi_{n+1})$ with $\fSA^*:\zstate \rightarrow\Re^d$, where $\zstate$ is the state space of Markov cha in $\bfPhi$.

\item[{\textbf{(A3)}}] The function $\fSA^*: \zstate \to \Re^d$ satisfies$\| {\fSA_i^*}^2\|_V<\infty$ and  $\pi(\fSA^*_i) = 0$ for each $i$.
\end{itemize}
%The reader is referred to \cite{MT} for definitions, except for a few clarifications and consequences: 

For any $g\in\LV$,  denote   $\tilg(z) = g(z) - \pi(g)$,  and 
\begin{equation}
\hag(z) = \sum_{n=0}^\infty   \Expect[\tilg(\Phi_n) \mid \Phi_0 =z]   
\label{e:fishSum}
\end{equation}
It is evident that   $\hag\in\LV$ under (A1).  Further conclusions are summarized below. \Theorem{t:Vuni}~(i) follows immediately from (A1).   Part~(ii) follows from (i) and \cite[Lemma 15.2.9]{MT}  (the chain is also $\sqrt{V}$-uniformly ergodic).

\begin{theorem}
\label{t:Vuni}
The following conclusions hold for a  $V$-uniformly ergodic Markov chain:  
\begin{romannum}
\item  The function $\hag\in\LV$ defined in \eqref{e:fishSum} has zero mean, and solves \textit{Poisson's equation}:
\begin{equation}
\Expect[\hag (\Phi_{k+1}) \! \mid \! \Phi_k \! = \! z]  \! = \! \hag(z)   -  \tilg(z)  
\label{e:fishDef}
\end{equation}

\item  
%The Markov chain is also  $V^\delta$-uniformly ergodic (see \cite[Chapter 16]{MT} for the definition) for any $\delta\in (0,1]$.   
% In particular,  
If $g^2\in\LV$,  then $\hag^2\in\LV$.  
\qed
\end{romannum}
\end{theorem}
%\ad{Removed $V^\delta$-uniform ergodicity from (iii).}

Assumption (A3) implies that the sequence $\{\Delta_n\}$ 
appearing in \eqref{e:SAlinearized}
is zero mean for the stationary version of the Markov chain $\bfPhi$. Its asymptotic covariance (appearing in the Central Limit Theorem) is denoted
\begin{equation}
\Sigma_\Delta  = \sum_{k=-\infty}^\infty  \Expect_\pi[\Delta_k\Delta_0^\transpose]
\label{e:SigmaDelta}
\end{equation}
where the expectations are in steady state.

A more useful representation of $\Sigma_\Delta$ is obtained through a decomposition of the noise sequence based on   Poisson's equation. This now standard technique was introduced in the SA literature in the 1980s  \cite{metpri84}.

With $\fSA^*$ defined in \eqref{e:f},  denote by $\haf$ a solution to   Poisson's equation:
\begin{equation}
\Expect[\haf\, ( \Phi_{k+1}) \mid \Phi_k =z] =\haf(z) -  \fSA^*(z) 
\label{e:DoubleFish}
\end{equation}
This is in fact $d$ separate Poisson equations since $\fSA^*\colon\zstate\to\Re^d$.  
It is assumed for convenience that the solutions are normalized so $\haf$ has zero steady-state mean. This is justified by the fact that $\haf-\pi(\haf)$ also solves  \eqref{e:DoubleFish} under assumption (A3). The fact that
$\haf_i^2 \in \LV$ for $1\le i \le d$ follows from \Theorem{t:Vuni}~(ii).   

We then write, for $n\ge 1$,
\[
\Delta_n = \fSA^*(\Phi_n) = \Delta_{n+1}^m  + Z_n   -  Z_{n+1}
\]
where $Z_n =  \haf(\Phi_n) $  and $ \Delta_{n+1}^m = Z_{n+1}   - \Expect[Z_{n+1}   \mid \clF_n] $ is a martingale difference sequence.    
Each of the sequences is bounded in $L_2$,  and the asymptotic covariance \eqref{e:SigmaDelta} is expressed
\begin{equation}
\Sigma_\Delta = \Expect_\pi[\Delta_n^m{\Delta_n^m}^\transpose]  
\label{e:SigmaDeltaFish}
\end{equation}
where the expectation is taken in steady-state.    The equivalence of \eqref{e:SigmaDeltaFish} and
\eqref{e:SigmaDelta} appears in  \cite[Theorem 17.5.3]{MT}  for the case in which $\Delta_n$ is scalar valued;  the generalization to vector valued processes involves only notational changes.   

%\sh{You said need assumptions here, we assume $\fSA^*$ is in $L_2$?
%\\
%SM2SH
%Obsolete?  If so, delete.}

\subsection{Decomposition and Scaling of the Parameter Sequence}
\label{s:decomp}

We now explain the decomposition \eqref{e:SAlinearized-decomp}. Each of the two sequences $\{\tiltheta_n^\clM\}$ and  $\{\tiltheta_n^\clT\}$ evolves as a stochastic approximation sequence, differentiated by the inputs and     initial conditions:
\begin{subequations} 
	\begin{align}
	\tiltheta_{n+1}^{\clM} 
	&=     \tiltheta_n^{\clM} 
	+ \alpha_{n+1} \bigl[ A  \tiltheta_n^{\clM}   \!  +   \Delta_{n+2}^m  \bigr]  \,,  && \tiltheta_0^{\clM} = \tiltheta_0
	\label{e:clEa-1}
	\\
	\tiltheta_{n+1}^{\clT} 
	&=     \tiltheta_n^{\clT} 
	+ \alpha_{n+1} \bigl[   A   \tiltheta_n^{\clT}  \!  +  \!  Z_{n+1} - Z_{n+2}  \bigr]\,, &&  \tiltheta_0^{\clT} =0
	\label{e:clEa-2}
	\end{align}%
	\label{e:decomp2}%
\end{subequations}
The second recursion admits a more tractable realization through the change of variables,  $\Xi_n =  \tiltheta_n^{\clT} + \alpha_n Z_{n+1}$,  $n\ge 1$.
\begin{lemma}
\label{t:Tele}
The sequence $\{\Xi_n\}$    evolves as the SA recursion
\begin{equation}
\Xi_{n+1} = 
\Xi_n +
 \alpha_{n+1} \bigl[   A \Xi_n    -    \alpha_n [I+A]    Z_{n+1}  \bigr]\,, \qquad \Xi_1  =  Z_{1}
\label{e:Tele}
\end{equation}
\qed
\end{lemma}
\vspace{-0.1in}

\Lemma{t:Tele} combined with  \eqref{e:decomp2} gives  
\begin{equation}
\tiltheta_n  = \tiltheta_n^{(1)}  +   \tiltheta_n^{(2)}  + \tiltheta_n^{(3)}
\label{e:unscaledDecomp}
\end{equation}
where  $\tiltheta_n^{(1)} \!= \tiltheta_n^\clM$, $\tiltheta_n^{(2)} \!= \Xi_n$, and $\tiltheta_n^{(3)} \!=\! -\alpha_nZ_{n+1} $  for $n\ge 1$. Note that $\tiltheta_n^\clT = \tiltheta_n^{(2)} + \tiltheta_n^{(3)}$.

It is more convenient to work directly with the recursion for the scaled sequence:
\begin{lemma}
\label{t:taylor-scale}
For any $\varrho \in (0,1/2]$, the scaled sequence $\tiltheta^{\varrho}_n =  n^\varrho \tiltheta_n$ admits the   recursion,
\begin{equation}
\tiltheta^{\varrho}_{n+1}  
\! = \! \tiltheta^{\varrho}_n  
+ \alpha_{n+1} \bigl[  \varrho_n \tiltheta^{\varrho}_n +   A(n,\varrho)   \tiltheta^{\varrho}_n
 +  (n+1)^\varrho \Delta_{n+1}  \bigr]   
\label{e:clElin}
\end{equation}
where  $\varrho_n = \varrho +\epsy(n,\varrho)$ with $\epsy(n,\varrho)=O(n^{-1})$, and $A(n,\varrho) =  (1+n^{-1})^\varrho A$.
\qed
\end{lemma}

Denote $\tiltheta_n^{\varrho,(i)} =  n^\varrho  \tiltheta_n^{(i)} $ for each $i$. 
\Lemma{t:taylor-scale} combined with \eqref{e:unscaledDecomp} gives  
\begin{equation}
\tiltheta^{\varrho}_n  =    \tiltheta_n^{\varrho,(1)}  +    \tiltheta_n^{\varrho,(2)}  +    \tiltheta_n^{\varrho,(3)}
\label{e:scaledDecomp}
\end{equation}
 The first two sequences evolve as SA recursions:
\begin{subequations}%
	\begin{align}
	 	  \tiltheta_{n+1}^{\varrho,(1)} 
	 =     \tiltheta_n^{\varrho,(1)} 
	+  & \alpha_{n+1} \bigl[  [  \varrho_n I + A(n,\varrho)]  \tiltheta_n^{\varrho,(1)} 
	   +  (n+1)^\varrho \Delta_{n+2}^m  \bigr]  \,, && \tiltheta_0^{\varrho,(1)} \!= \! \tiltheta^{\varrho}_0
	\label{e:clE-1}
	\\
	   \tiltheta_{n+1}^{\varrho,(2)} 
	=     \tiltheta_n^{\varrho,(2)} 
	+ & \alpha_{n+1}   \bigl[    [  \varrho_n I + A(n,\varrho)]  \tiltheta_n^{\varrho,(2)}  
	 -  (n+1)^\varrho    \alpha_n [I+A]    Z_{n+1} 
	\bigr] \,, && \tiltheta^{\varrho,(2)}_1 \!= \!  \Xi_1
	\label{e:clE-2}
	\end{align}%
	\label{e:auto2}%
\end{subequations}%
%\ad{Ana suggests we make the section heading MSE bounds?}
\subsection{Mean Square Error Bounds}
\label{s:mse}

 Fix the initial condition $(\Phi_0, \tiltheta_{0})$, and denote
 $\Cov(\theta_n) = \Expect[\tiltheta_n\tiltheta_n^\transpose]$ and $\Sigma_Z = \Expect_{\pi}[Z_n Z_n^\transpose]$. 
 The following summarizes bounds on the convergence rate of $\Expect[\|\tiltheta_n\|^2] = \trace(\Cov(\theta_n))$.
 % \rd{An eigenvalue $\lambda$ of $A$ is denoted $\lambda = -\varrho_0 + ui$. }
 
% We say that a deterministic sequence $\{\clE_n\}$ converges to zero at rate $1/n^\varrho$, if
%\begin{equation}
%\lim_{n\to\infty  } n^{\varrho'} \|\clE_n\|
%=
%\begin{cases}
%0, & \text{if } \;\varrho' < \varrho \\
%\infty, & \text{if }\; \varrho' > \varrho
%\end{cases}
%\end{equation}
 \begin{theorem}
 \label{t:SAlinearized}
 Suppose {(A1)}-{(A3)} hold. Then, for the linear recursion \eqref{e:SAlinearized},
 \begin{romannum}
 \item  If $\text{Real}(\lambda)<-\half$ for every eigenvalue $\lambda$ of $A$, then for some  $\delta=\delta(A,\Sigma_\Delta)>0$, 
 \[
 \Cov(\theta_n) = n^{-1} \Sigma_\theta  +  O( n^{-1-\delta})  \,,\qquad n\ge 0\,,
 \]
 where $\Sigma_\theta \geq 0$ is the solution to the Lyapunov equation \eqref{e:lyap}. Consequently, the rate of convergence of $\Expect[\|\tiltheta_n\|^2]$ is $1/n$.

 \item  Suppose there is an eigenvalue $\lambda$ of $A$ that satisfies {$-\varrho_0 = \text{Real}(\lambda)> -\half$}. Let $v \neq0$ denote a corresponding left eigenvector, 
 and suppose that $\Sigma_\Delta v \neq 0$.   Then, $\Expect[ |v^\transpose \tiltheta_n|^2]  $ converges to $0$ at  rate $n^{-2\varrho_0}$.
% \begin{equation}
% \label{e:vclE}
% \lim_{n\to\infty  } n^{2\varrho} \Expect[ |v^\transpose \tiltheta_n|^2]  
% =
% \begin{cases}
% 0, & \text{if } \;\varrho < \varrho_0 \\
% \infty, & \text{if }\; \varrho > \varrho_0
% \end{cases}
% \end{equation}
% Consequently, the rate of convergence to zero of $\Expect[\|\tiltheta_n\|^2]$ can be no faster than $n^{-\varrho_0}$.
 \qed
 \end{romannum}
 \end{theorem}

The proof of \Theorem{t:SAlinearized} is contained in \Section{s:proof_finite_n}.    
The following negative result is a direct corollary of \Theorem{t:SAlinearized} (ii):
\spm{Delete this paragraph for AISTATS}
\begin{corollary}
Suppose {(A1)}-{(A3)} hold. Moreover, suppose there is an eigenvalue $\lambda$ of $A$ that satisfies {$-\varrho_0 = \text{Real}(\lambda)> -\half$}, with corresponding left eigenvector $v$ satisfying $\Sigma_\Delta v \neq 0$. Then, $\Expect[\|\tiltheta_n\|^2]$ converges to zero at rate no faster than $1 / n^{2\varrho_0}$.
\end{corollary}
%\ad{Is this very obvious? I understand it is, but can we add a one-line explanation?}

\smallskip

\smallskip

One challenge in extension to nonlinear recursions is that the noise sequence depends on the parameter estimates (recall \eqref{e:linSA}).   This is true even for TD learning with linear function approximation (see \eqref{e:TD0_linSA_err} and surrounding discussion).   Extension to these recursions is obtained through coupling. 

Consider the error sequence for a random linear recursion 
\begin{equation}
\tiltheta^\circ_{n+1} = \tiltheta^\circ_n + \alpha_{n+1}[A_{n+1}\tiltheta^\circ_n + A_{n+1}\theta^*  - b_{n+1}] 
\label{e:linSAerr_An}
\end{equation}
subject to the following assumptions:  
\begin{itemize}
	\item[\textbf{(A4)}]
	The sequences $\{A_n,b_n\}$ are functions of the Markov chain: 
\[
A_n = \Amap(\Phi_n) \,,\quad 
b_n = \Bmap(\Phi_n) \,,  
\]
which satisfy  $\| \Amap_{i,j}^2\|_V<\infty$,  $\| \Bmap_i^2\|_V<\infty$ for each $1\le i,j \le d$.    The steady state means  are denoted $A = \Expect_\pi[A_n]$,  $b = \Expect_\pi[b_n]$.  Moreover, the matrix $A$ is Hurwitz, and $\theta^* = A^{-1} b$. 
\end{itemize}

\begin{theorem}
\label{t:TD}
Suppose Assumptions A1-A4 hold. If the matrix $\half I + A$ is Hurwitz, the error bound \eqref{e:finite-n1} holds for $\{\tiltheta^\circ_n\}$ obtained from \eqref{e:linSAerr_An}, with $\Delta_{n+1} = A_{n+1}\theta^* - b_{n+1}$.
%\begin{romannum}
%	\item Suppose there exists an eigenvalue that satisfies $-\varrho_0 = Real(\lambda) > -\half$. Let $v^\transpose \neq 0$ denote the corresponding left eigenvector, and suppose that $\Sigma_{\Delta}v\neq 0$. Then,  the convergence rate $n^{-2\varrho_0}$ holds for $\Expect[\|\tiltheta^\circ_n\|^2]$. 
%\end{romannum}
\end{theorem}

The proof of the theorem is via coupling with \eqref{e:SAlinearized}.     For this we write \eqref{e:linSAerr_An} in the suggestive form
\begin{equation}
\tiltheta^\circ_{n+1} = \tiltheta^\circ_n + \alpha_{n+1}[A\tiltheta^\circ_n + \Delta_{n+1}  + (A_{n+1}-A)\tiltheta^\circ_n]\,,\qquad \Delta_{n+1} = A_{n+1}\theta^* - b_{n+1} 
\label{e:TD0_linSA_err-2}
\end{equation}
With common initial condition  $\Phi_0$,  the sequence $\{\tiltheta^\circ_n\}$  is compared with the   error sequence $\{\tiltheta^\bullet_n\}$ for the corresponding  linear SA algorithm:
\[
\tiltheta^\bullet_{n+1} = \tiltheta^\bullet_n + \alpha_{n+1}[A\tiltheta^\bullet_n + \Delta_{n+1}]
\]

The difference sequence $\{\clE_n \eqdef \tiltheta^\circ_n - \tiltheta^\bullet_n\}$ evolves according to \eqref{e:SAlinearized}, but with a vanishing noise sequence:
\begin{equation}
\clE_{n+1} = \clE_n + \alpha_{n+1}[A\clE_n + (A_{n+1}-A)\tiltheta^\circ_n]
\label{e:clE-td-err}
\end{equation}
By decomposing $A_{n+1}-A$ into martingale difference  and telescoping sequences based on Poisson's equation,   the technique used to prove \Theorem{t:SAlinearized} can be used to obtain the following bound on the mean-square coupling error. %The   proof is contained in Appendix~\ref{s:couple-td}. 

Let  $\lambda = -\varrho_0 + ui$  denote an eigenvalue of the matrix $A$ with largest real part (i.e., $\varrho_0$ is minimal).
\begin{theorem}
\label{t:couple-td}
Under Assumptions (A1)-(A4), 
\begin{romannum}
\item   
$\displaystyle 
\limsup_{n\rightarrow \infty} n^2\Expect[\clE_n^\transpose\clE_n] < \infty
$   if $\varrho_0>1$.

\item $\displaystyle 
  \limsup_{n\rightarrow\infty} n^{2\varrho} \Expect[\clE_n^\transpose\clE_n] <\infty $  
for all $\varrho  < \varrho_0 $,   provided $\varrho_0 \le 1$.
\end{romannum} 
\qed
\end{theorem}

\Theorem{t:couple-td}  provides a remarkable bound when $\rho_0>1$:  it immediately implies  \Theorem{t:TD} because the mean-square coupling error $\Expect[\|\clE_n]\|^2$ tends to zero at rate no slower than $n^{-2}$,  which is far faster than  $\Expect[\|\tiltheta^\bullet_n\|^2] \approx \trace(\Sigma_\theta) n^{-1}   $  (implied by  \Theorem{t:SAlinearized}).

An alert reader may observe that Theorems~\ref{t:TD} and \ref{t:couple-td} leave out a special case:   consider $\rho_0<\half$,  so that the rate of convergence of $\Expect[\|\tiltheta^\bullet_n\|^2]$ is the sub-optimal value $n^{-2\rho_0}$.   The bound obtained in \Theorem{t:couple-td} remains valuable, in the sense that it combined with  \Theorem{t:SAlinearized}~(ii) implies the rate of convergence of $\Expect[\|\tiltheta^\circ_n]\|^2$  is no slower than $n^{-2\rho_0}$.  However, because $\Expect[\|\clE_n]\|^2$  and $\Expect[\|\tiltheta^\bullet_n\|^2] $ tend to zero at the same rate,   we cannot rule out the possibility that  $ \tiltheta^\circ_n =  \clE_n  + \tiltheta^\bullet_n$ converges to zero much faster.    In particular, it remains to prove that  if there is an eigenvalue $\lambda$ of $A$ that satisfies {$-\varrho_0 = \text{Real}(\lambda)> -\half$},  and an eigenvector $v \neq0$ satisfying $\Sigma_\Delta v \neq 0$,  then, $\Expect[ |v^\transpose \tiltheta_n^\circ|^2]  $ converges to $0$ at  rate $n^{-2\varrho_0}$.

\subsection{Implications}
\label{s:imp}

\Theorem{t:SAlinearized} indicates that the convergence rate of $\Cov(\theta_n)$ is determined jointly by the matrix $A$, and the martingale difference component of the noise sequence $\{\Delta_n\}$. Convergence of $\{\tiltheta_n\}$ can be slow if the matrix $A$ has eigenvalues close to zero. 

The result also explains the slow convergence of some reinforcement learning algorithms. For instance, the matrix $A$ in Watkins' Q-learning has at least one eigenvalue with real part greater than or equal to $-(1-\beta)$, where $\beta$ is the discount factor appearing in the Markov decision process \cite{wat89,devmey17a,dev19}. Since $\beta$ is usually close to one, \Theorem{t:SAlinearized} implies that the convergence rate of the algorithm is much slower than $n^{-1}$.  Under the assumption that $A$ is Hurwitz, the $1/n$ convergence rate is guaranteed by the use of a modified step-size sequence $\alpha_n = g/n$, with $g > 0$ chosen so that the matrix $\half I + gA$ is Hurwitz. \Corollary{c:main_g} follows directly from \Theorem{t:SAlinearized} (i).

\begin{corollary}
\label{c:main_g}
Let $g$ be a constant such that $\half I + gA$ is Hurwitz, and $\{\tiltheta_n\}$ be recursively obtained as
\[
\tiltheta_{n+1}  = \tiltheta_n + \frac{g}{n+1}[A \tiltheta_n  + \Delta_{n+1}  ] 
\]
Then,  for some $\delta=\delta(A,g, \Sigma_\Delta)>0$,
\[
\Cov(\theta_n) = \Expect[\tiltheta_n\tiltheta_n^\transpose] = n^{-1}\Sigma_\theta^g + O(n^{-1-\delta})
\]
where $\Sigma_\theta^g \geq 0$ solves the Lyapunov equation
\[
[\half I + gA]\Sigma + \Sigma [\half I + gA]^\transpose + g^2\Sigma_{\Delta}=0
\]
\qed
\end{corollary}
We can also ensure the $1/n$ convergence rate by using a matrix gain. Provided $A$ is invertible, and if it is known beforehand, $\alpha_n = -A^{-1} / n$ is the optimal step-size sequence (in terms of minimizing the asymptotic covariance)~\cite{benmetpri90,kusyin97,devbusmey20}. The SQNR algorithm of \cite{rup85} and the Zap-SNR algorithm \cite{devmey17a,dev19} provide general approaches to recursively estimate the optimal matrix gain.

The next subsection is dedicated to the proof of \Theorem{t:SAlinearized}. The proofs of the technical results are contained in the Appendix \ref{s:app-msr}.

\subsection{Proof of  \Theorem{t:SAlinearized}}
\label{s:proof_finite_n}

Denote $\Cov(\theta_n^{(i)}) = \Expect[\tiltheta_n^{(i)}(\tiltheta_n^{(i)})^\transpose]$ and $\Sigma_n^{\varrho,(i)} =\Expect[\tiltheta^{\varrho,(i)}(\tiltheta^{\varrho,(i)})^\transpose] = n^{2\varrho}\Cov(\theta_n^{(i)})$ for each $i$ in \eqref{e:scaledDecomp}. The proof proceeds by establishing the convergence rate for each $\Cov(\theta_n^{(i)})$. The main challenges are the first two: $\Cov(\theta_n^{(1)})$ and $\Cov(\theta_n^{(2)})$, for which explicit bounds are obtained by studying recursions of the scaled sequences. 
Bounding $\tiltheta_n^{{(3)}} = -\alpha_n Z_{n+1}$ is trivial.

\subsubsection{The martingale difference term}  

\begin{proposition} 
\label{t:mart-3}
Under  (A1)-(A3),
\begin{romannum}
\item If $\text{Real}(\lambda)<-\half$ for every eigenvalue $\lambda$ of $A$, then
\[
\Cov(\theta_n^{(1)}) = n^{-1}\Sigma_\theta + O(n^{-1-\delta})
\]
where $\delta = \delta(\half I + A,\Sigma_\Delta) >0$, and $\Sigma_\theta$ is the solution to the Lyapunov equation \eqref{e:lyap}.

\item Suppose there is an eigenvalue $\lambda$ of $A$, that satisfies {$-\varrho_0 = \text{Real}(\lambda)> -\half$}. Let  $v \neq 0 $ denote the corresponding left eigenvector, and suppose moreover that $ \Sigma_\Delta v \neq 0$. Then, $\Expect[  | v^\transpose \tiltheta_n^{(1)} |^{2}] $ converges to $0$ at rate $n^{-2\varrho_0}$.
%\[ 
%\begin{aligned}
%\limsup_{n\to\infty}  n^{2\varrho_0} \Expect[  | v^\transpose \tiltheta_n^{(1)} |^{2}] <\infty
%\end{aligned}
%\]
%Consequently, $\lim_{n\to\infty} n^{2\varrho} \Expect[  | v^\transpose \tiltheta_n^{(1)} |^{2}] =0$ for $\varrho < \varrho_0$.
%\sh{The $\liminf > 0$ comes from the old proof. Not needed here since the unstable result is proved independently.}
\end{romannum}
\qed
\end{proposition}

\subsubsection{The telescoping sequence term} 

\begin{proposition}
\label{t:tele-3}
Under (A1)-(A3),
\begin{romannum}

\item If $\text{Real}(\lambda)<-\half$ for every eigenvalue $\lambda$ of $A$, then,  $\Cov(\theta_n^{(2)}) = O(n^{-1-\delta})$ for some $\delta = \delta(\half I + A,\Sigma_\Delta)>0$.

\item Suppose there is an eigenvalue $\lambda$ of $A$ that satisfies {$-\varrho_0 = \text{Real}(\lambda)> -\half$}. Let  $v \neq 0 $ denote the corresponding left eigenvector, and suppose moreover that $ \Sigma_\Delta v \neq 0$. Then,
\[
\limsup_{n\rightarrow \infty} n^{2\varrho_0}\Expect[|v^\transpose \tiltheta_n^{(2)}|^2] < \infty
\]
%Consequently, $\lim_{n\rightarrow \infty} n^{2\varrho} \Expect[|v^\transpose \tiltheta_n^{(2)}|^2]=0$ for any $\varrho < \varrho_0$.
%\sh{lim = 0 is not absolutely necessary. But we want to use this directly to show Thm 2.4 (ii) $\lim =0$ case.}
\end{romannum}
\qed
\end{proposition}

\subsubsection{Proof of  \Theorem{t:SAlinearized}}

We obtain the convergence rate of $\Cov(\theta_n)$ based on
\[
\Cov(\theta_n) = \sum_{i=1}^3 \Cov(\theta_n^{(i)}) + \sum_{i=1}^3\sum_{j=1, j\neq i}^3 \Expect[\tiltheta_n^{(i)} (\tiltheta_n^{(j)})^\transpose]
\]

For case (i), by \Proposition{t:mart-3} (i) and \Proposition{t:tele-3} (i), there exists 
$\delta = \delta(\half I +A,\Sigma_\Delta)>0 $ such that
\[ 
\begin{aligned}
\Cov(\theta_n^{(1)})  &= n^{-1}\Sigma_\theta + O(n^{-1-\delta}) \\
\Cov(\theta_n^{(2)}) &= O(n^{-1-\delta }) \\
\Cov(\theta_n^{(3)}) & = n^{-2} \Sigma_{Z_{n+1}} 
\end{aligned}
\]
The cross terms between $\tiltheta_n^{(i)}$ and $\tiltheta_n^{(j)}$ for $i\neq j$ are of smaller orders than $O(1/n)$ by the Cauchy-Schwarz inequality. Therefore, for a possibly smaller $\delta>0$,
\[
\Cov(\theta_n) = n^{-1}\Sigma_\theta + O(n^{-1-\delta})
\]
For case (ii),  $\lim_{n\rightarrow 0}n^{2\varrho}\Expect[|v^\transpose \tiltheta_n|^2] =0$  for each $\varrho < \varrho_0$ can be obtained from \Proposition{t:mart-3} (ii) and \Proposition{t:tele-3} (ii) directly by the triangle inequality. For $\varrho > \varrho_0$, the result $\lim_{n\rightarrow 0}n^{2\varrho}\Expect[|v^\transpose\tiltheta_n|^2] =\infty$ is established independently in \Lemma{t:unstable}.
\qed

\subsection{Finer Error Bound}
\label{s:fine-error}
\subsubsection{Finer Decomposition with Second Poisson Equation}

With $\haf$ in \eqref{e:DoubleFish} and  that $\haf_i^2\in \LV$ for each $1\le i \le d$, denote $\hahaf$ by the zero-mean solution to the second Poisson equation
\begin{equation}
\Expect[\hahaf\, ( \Phi_{k+1}) \mid \Phi_k =z] =\hahaf(z) -  \haf(z) 
\label{e:DoubleFish-2}
\end{equation}
We then write, for $n\ge 1$,
\begin{equation}
\label{e:DoubleFish-2}
Z_n = \haDelta_{n+1}^{m} + \haZ_n - \haZ_{n+1}
\end{equation}
where $\haZ_n = \hahaf(\Phi_n)$,  and  $\haDelta_{n+1}^m = \haZ_{n+1} - \Expect[\haZ_{n+1}\mid \clF_n]$ is a martingale difference sequence.

The type of decomposition in Section \ref{s:decomp} can be applied to $\tiltheta_n^{(2)}$ in \eqref{e:Tele} for $n\ge 2$:
\begin{equation}
\tiltheta_n^{(2)}  = \tiltheta_n^{(2,1)}  +   \tiltheta_n^{(2,2)}  + \tiltheta_n^{(2,3)}
\label{e:unscaledDecomp-finer-2}
\end{equation}
The first two sequences evolve as SA recursions:
\begin{subequations}%
	\begin{align}
	  \tiltheta_{n+1}^{(2, 1)} 
	=     \tiltheta_n^{(2, 1)} 
	+  & \alpha_{n+1} \bigl[    A  \tiltheta_n^{(2,1)} 
	- \alpha_n [I+A] \haDelta_{n+2}^m  \bigr] \,, && \tiltheta_1^{(2,1)} = Z_1
	\label{e:clE-finer-1}
	\\
	  \tiltheta_{n+1}^{(2, 2)} 
	=     \tiltheta_n^{(2, 2)} 
	+ & \alpha_{n+1}   \bigl[   A  \tiltheta_n^{(2, 2)}  
	+    \alpha_{n-1}\alpha_n [2I+A][I+A]    \haZ_{n+1}
	\bigr] \,, && \tiltheta_2^{(2, 2)} = -\half [I+A]\haZ_{2}
	\label{e:clE-finer-2}
	\end{align}%
	\label{e:auto2-finer}%
\end{subequations}%
and $\tiltheta_n^{(2,3)} = \alpha_{n-1}\alpha_n[I+A] \haZ_{n+1}$. Therefore, $\tiltheta_n$ for $n\ge 2$ can be decomposed as:
\begin{equation}
\tiltheta_n  = \tiltheta_n^{(1)}  +   \tiltheta_n^{(2,1)} + \tiltheta_n^{(2,2)} +\tiltheta_n^{(2,3)}  + \tiltheta_n^{(3)}
\label{e:cle-finer-decomp}
\end{equation}

\subsubsection{Finer Mean Square Error Bound}

The   error bound \eqref{e:finite-n2} is obtained from \eqref{e:cle-finer-decomp}:

 \begin{theorem}
\label{t:SAlinearized-finer}
Suppose Assumptions {(A1)}-{(A3)} hold, and moreover $\text{Real}(\lambda) <-1$ for every eigenvalue $\lambda$ of $A$. Then, for the linear recursion \eqref{e:SAlinearized},
\[
\Cov(\theta_n) = n^{-1}\Sigma_\theta +  n^{-2}\Sigma_{\theta,2}  + O(n^{-2-\delta}) 
\]
where $\delta=\delta(I+A,\Sigma_\Delta)>0$, $\Sigma_{\theta,2}=\Sigma_{\sharp} + \Sigma_Z -\Expect_{\pi}[\Delta_{n}^m \haZ_{n}^\transpose] -\Expect_{\pi}[\haZ_{n}(\Delta_{n}^m)^\transpose]$, and $\Sigma_{\sharp} $ is the unique solution to the Lyapunov equation:
\begin{equation}
[I+A][\Sigma-\Cov_\pi(\haDelta_n^m, \Delta_n^m)]+ [\Sigma -\Cov_\pi(\Delta_n^m, \haDelta_n^m)][I+A]^\transpose + A\Sigma_\theta A^\transpose  - \Sigma_\Delta = 0
\label{e:SigLya-finer}
\end{equation}
\qed
\end{theorem}

\subsubsection{Proof of \Theorem{t:SAlinearized-finer}}

Denote the correlation between $\tiltheta_n^{(a)}$ and $\tiltheta_n^{(b)}$ as $R_n^{(a),(b)} = \Expect[\tiltheta_n^{(a)}(\tiltheta_n^{(b)})^\transpose]$, where $\tiltheta_n^{(a)}, \tiltheta_n^{(b)}$ are different terms in \eqref{e:cle-finer-decomp}. The key results that help establish \Theorem{t:SAlinearized-finer} are summarized in the following proposition. The proof is in Appendix \ref{s:app-finer}.
\begin{proposition}
\label{t:fine-comp}
Under Assumptions (A1)-(A3), if $\text{Real}(\lambda) < -1$ for every eigenvalue of $A$, then there is $\delta>0$ such that 
\begin{romannum}
\item
$
\Cov(\theta_n^{(1)}) = n^{-1}\Sigma_\theta + n^{-2}\Sigma_{\sharp}^{(1)} + O(n^{-2-\delta})
$,
where $\delta = \delta(I+A, \Sigma_\Delta)>0$, $\Sigma_\theta \geq 0 $ is the unique solution to the Lyapunov equation \eqref{e:lyap}, and $\Sigma_{\sharp}^{(1)}\geq 0$ solves the Lyapunov equation,
\begin{equation}
[I+A]\Sigma + \Sigma[I+A]^\transpose + A\Sigma_\theta A^\transpose  - \Sigma_\Delta = 0
\label{e:SigLya-2}
\end{equation}
\item  
$
R_n^{(2,1),(1)}  + R_n^{(1), (2,1)} = n^{-2}\Sigma_{\sharp}^{(2)} + O(n^{-2-\delta})
$,
where $\Sigma_{\sharp}^{(2)}$ solves the Lyapunov equation:
\begin{equation}
[I+A]\Sigma + \Sigma[ I+A]^\transpose  - [I+A]\Cov_\pi(\haDelta_n^m, \Delta_n^m) - \Cov_\pi(\Delta_n^m, \haDelta_n^m)[I+A]^\transpose= 0
\label{e:SigLya-c}
\end{equation}
\item 
$ 
R_n^{(1),(3)} = -n^{-2}\Expect_{\pi}[\Delta_{n}^m \haZ_{n}^\transpose] + O(n^{-3})$.
\end{romannum}
\qed
\end{proposition}

\begin{proof}[Proof of \Theorem{t:SAlinearized-finer}]
With the decomposition in \eqref{e:cle-finer-decomp}, we have
\[
\begin{aligned}
\Cov(\theta_n) = \Cov(\theta_n^{(1)}) &+ \sum_{j=1}^3 \Cov(\theta_n^{(2,j)} ) +  \Cov(\theta_n^{(3)})   + R_n^{(1),(3)} + R_n^{(3),(1)}  \\
&+\sum_{i\in\{1,3\}}\sum_{j=1}^3 [R_n^{(2,j),(i)} +R_n^{(i),(2,j)}  ] 
 + \sum_{j=1}^3\sum_{k=1, k\neq j}^3 [R_n^{(2,j),(2,k)} + R_n^{(2,k),(2,j)}]
\end{aligned}
\]
 $\Cov(\theta_n^{(2,1)} )= O(n^{-3})$, $\Cov(\theta_n^{(2,2)}) = O(n^{-5})$ and $\Cov(\theta_n^{(2,3)} ) = O(n^{-4})$ by \Theorem{t:SAlinearized} (i).   
By the Cauchy-Schwarz inequality, the correlation terms involving $\tiltheta_n^{(2,2)}$ and $\tiltheta_n^{(2,3)}$ are $O(n^{-2.5})$, and $ R_n^{(2,1),(3)} = O(n^{-2.5})$ is also $O(n^{-2.5})$. \Prop{t:fine-comp} (ii) shows that $R_n^{(2,1),(3)} = O(n^{-3})$. Hence the covariance can be approximated as follows:
\[
\Cov(\theta_n) = \Cov(\theta_n^{(1)})  +  \Cov(\theta_n^{(3)}) + R_n^{(1),(3)} + R_n^{(3),(1)} +  R_n^{(2,1),(1)}+ R_n^{(1),(2,1)} + O(n^{-2.5})
\]
By \Proposition{t:fine-comp}, there exist $\delta(I+A,\Sigma_{\Delta})>0$ and $\delta(I+A)>0$ such that
\[
\begin{aligned}
\Cov(\theta_n^{(1)})                                                             & =  n^{-1}\Sigma_\theta +               n^{-2}\Sigma_{\sharp}^{(1)} + O(n^{-2-\delta}) \\
\Cov(\theta_n^{(3)})                                                            & = n^{-2}\Sigma_Z + O(\rho^n)  \\
R_n^{(1),(3)}   														   & =-n^{-2}\Expect_{\pi}[\Delta_{n}^m \haZ_{n}^\transpose] + O(n^{-3}) \\
R_n^{(2,1),(1)}+ R_n^{(1),(2,1)} 							&=n^{-2}\Sigma_{\sharp}^{(2)} + O(n^{-2-\delta})
\end{aligned}
\]
Putting those results together gives 
\[
\Cov(\theta_n) =  n^{-1}\Sigma_\theta +  n^{-2}\bigl(\Sigma_{\sharp}^{(1)}  + \Sigma_{\sharp}^{(2)} +\Sigma_Z -\Expect_{\pi}[\Delta_{n}^m \haZ_{n}^\transpose] -\Expect_{\pi}[\haZ_{n}(\Delta_{n}^m)^\transpose] \bigr) + O(n^{-2-\delta}) 
\]
for some $\delta>0$,   where  $\Sigma_{\sharp}\eqdef \Sigma_{\sharp}^{(1)} + \Sigma_{\sharp}^{(2)}$ solves the Lyapunov equation~\eqref{e:SigLya-finer}.
\end{proof}

\section{Conclusions}

Performance bounds for recursive algorithms are challenging outside of the special cases surveyed in the introduction.   The general framework developed in this paper provides tight finite time performance for linear stochastic recursions under mild conditions on the Markovian noise,  and we are confident that the techniques will extend to obtain similar bounds for nonlinear stochastic approximation provided that the linearization \eqref{e:linSA} is meaningful.

The bound \eqref{e:finite-n1} implies that, for some constant $b_\theta$ and all $n$,
\[
\Expect[\|\tiltheta_n\|^2]  \le   n^{-1} \trace(\Sigma_\theta)  + n^{-1-\delta} b_\theta  \, .
\]
It may be argued that we have not obtained a finite-$n$ bound, because a bound on the constant $b_\theta$   is lacking.    Our response is that the precision of the dominant term is most important.   We have tested the bound in numerous experiments in which the empirical mean-square error   is obtained from multiple independent trials,  and the resulting histogram is compared to what is predicted by the Central Limit Theorem with covariance $\Sigma_\theta$.   It is  found that the Central Limit Theorem is highly predictive of finite-$n$ performance in most cases  \cite{devmey17a,dev19,devbusmey20}.  While it  is hoped that further research will provide bounds on $b_\theta$,   it seems likely that any bound   will   involve high-order statistics of the Markov chain;
evidence of this is the complex coefficient of $n^{-2}$ in  \eqref{e:finite-n2} for the special case   $\delta=1$.

Current research concerns these topics, as well as algorithm design for reinforcement learning in various settings.

\bibliographystyle{abbrv}
\bibliography{strings,markov,q,NIPS19extras,SA_finite_N_extras}

\clearpage
\appendix

\section{Appendices}
\label{s:app-msr}

\subsection{Proofs for decomposition and scaling}

\begin{proof}[Proof of \Lemma{t:Tele}]
	Recall the summation by parts formula:  for scalar sequences $\{a_k,b_k\}$,
\begin{equation}
	\sum_{k=0}^N  a_{k+1} [b_{k+1}-b_k] =    
	a_{k+1} b_{k+1} - a_1b_0   
	-
	\sum_{k=1}^N  [a_{k+1} - a_k] b_k 
	\label{e:SbP}
\end{equation}

	This is applied to \eqref{e:clEa-2}, beginning with  
\[
	\tiltheta_{N+1}^{\clT}    
	=  \sum_{n=0}^N  \alpha_{n+1}  A  \tiltheta_n^{\clT}     +   \sum_{n=0}^N       \alpha_{n+1} [ Z_{n+1} - Z_{n+2} ]  
\]
	Hence with $a_k =  \alpha_k $  and $b_k =  Z_{k+1} $,  the identity \eqref{e:SbP}  implies  
\[
	\begin{aligned} 
	\sum_{n=0}^N       \alpha_{n+1} [ Z_{n+1} - Z_{n+2} ]  
	&=   
	Z_1 - \alpha_{N+1} Z_{N+2}
	+    
	\sum_{n=1}^N     [  \alpha_{n+1}  - \alpha_n ]  Z_{n+1}  
	\\  
	&
	=
	Z_1 
	- \alpha_{N+1} Z_{N+2}
	-   
	\sum_{n=1}^N       \alpha_{n+1} \alpha_n   Z_{n+1}    
	\end{aligned} 
\]
	By substitution, and using $ \tiltheta_0^{\clT}    =0$,
\[
	\tiltheta_{N+1}^{\clT}    
	=  Z_1 
	- \alpha_{N+1} Z_{N+2}
	+ \sum_{n=1}^N  \alpha_{n+1}   \bigl[   A  \tiltheta_n^{\clT}     - \alpha_n   Z_{n+1}   \bigr]
\]
	
	With $
	\Xi_n \eqdef  \tiltheta_n^{\clT}    + \alpha_n Z_{n+1}$ for $n\ge 1$ we finally obtain  for $N\ge 1$,
\[
	\Xi_{N+1} 
	%\tiltheta_{N+1}^{(3)}    + \alpha_{N+1} Z_{N+2}
	=       Z_1 + \sum_{n=1}^N  \alpha_{n+1}   \bigl[   A \Xi_n   -  \alpha_n [I+A]    Z_{n+1}  \bigr]   
	% A  [\Xi_n - \alpha_n Z_{n+1} ]    -    \alpha_n   Z_{n+1}     
\]
	which is equivalent to \eqref{e:Tele}.
\end{proof}

\begin{proof}[Proof of \Lemma{t:taylor-scale}]
	Consider the Taylor series expansion:
\[
	\begin{aligned} 
	\frac{(n+1)^\varrho}{n^\varrho} = (1+n^{-1})^\varrho 
	&= 1 + \varrho n^{-1} - \half \varrho(1-\varrho) n^{-2}  +O(n^{-3})
	\\
	&  = 1 +  \varrho (n+1)^{-1}  + \varrho  n^{-1}  (n+1)^{-1}   - \half \varrho(1-\varrho) n^{-2}  +O(n^{-3})
	\end{aligned} 
\]
	where the second equation uses $n^{-1} - (n+1)^{-1} =   n^{-1}  (n+1)^{-1} $.  With $\alpha_n=1/n$, the following bound follows:
\[
	(n+1)^\varrho  =  n^\varrho  \bigl[ 1 + \alpha_{n+1}(\varrho + \epsy(n,\varrho)  ) \bigr]
\]
	where $\epsy(n,\varrho) = O( n^{  -1})$,   and $ \epsy(n,\varrho) >0$ for all $n$. 
	
	Multiplying both sides of \eqref{e:SAlinearized} by $(n+1)^\varrho$, we obtain 
\[
	\tiltheta^{\varrho}_{n+1}  
	=      \tiltheta^{\varrho}_n  
	+ \alpha_{n+1} \bigl[  \varrho_n \tiltheta^{\varrho}_n +   A(n,\varrho)   \tiltheta^{\varrho}_n  +  (n+1)^\varrho \Delta_{n+1}  \bigr]   
\]
	where  $\varrho_n = \varrho +\epsy(n,\varrho) $  and $A(n, \varrho) =  (1+n^{-1})^\varrho A$.
\end{proof}

\begin{lemma}
\label{t:prod-n}
Let $\varrho_0 >0, L\geq 0$ be fixed real numbers. Then  the following holds for each $n\geq 1$ and   $1\leq n_0< n$:
\[
\prod_{k=n_0}^{n} [1-\varrho_0\alpha_k + L^2 \alpha_k^2]\leq K_{\ref{t:prod-n}}  \frac{n_0^{\varrho_0}}{(n+1)^{\varrho_0}}
\]
where $K_{\ref{t:prod-n}} = \exp( \varrho_0+ L^2\sum_{k=1}^\infty \alpha_k^2)$.
\end{lemma}
\begin{proof}
By the inequality $1-x \leq \exp(-x)$,
\[
\prod_{k=n_0}^{n} [1-\varrho_0\alpha_k + L^2 \alpha_k^2] 
	\leq \exp(-\varrho_0\sum_{k=n_0}^n \alpha_k)\exp( L^2\sum_{k=n_0}^n \alpha_k^2)
	\leq \exp( -\varrho_0) K \exp(-\varrho_0\sum_{k=n_0}^n \alpha_k) 
\]
The remainder of the proof involves establishing the bound
\begin{equation}
 \exp(-\varrho_0\sum_{k=n_0}^n \alpha_k)   \le \exp(\varrho_0)\frac{n_0^{\varrho_0}}{(n+1)^{\varrho_0}}
\label{e:prod-n-pf}
\end{equation}
For $n_0=1$ this follows from the bound  $\sum_{k=1}^n \alpha_k \ge  \ln(n+1)$,   and for $n_0\ge 2$ the bound \eqref{e:prod-n-pf} follows from
   $\sum_{k=n_0}^n \alpha_k > \ln(n+1) - \ln(n_0-1) -1$.
\end{proof}

\begin{lemma}
\label{t:converge-sa}
	Under Assumptions A1-A3, let $\lambda = -\varrho_0 + ui$ denote an eigenvalue of matrix $A$ with largest real part. Then
\[
	\lim_{n\rightarrow\infty} n^{2\varrho} \Expect[\tiltheta_n^\transpose\tiltheta_n] = 0\,, \qquad \varrho < \varrho_0\text{ and } \varrho \leq \half
\]
\end{lemma}

\begin{proof}
	Recall the decomposition of $\tiltheta_n$ in \eqref{e:unscaledDecomp}: $\tiltheta_n = \tiltheta_n^{(1)} + \tiltheta_n^{(2)} + \tiltheta_n^{(3)}$, with $\tiltheta_n^{(1)}, \tiltheta_n^{(2)}$ evolving as 
	\begin{subequations} 
		\begin{alignat}{3}
		& \tiltheta_{n+1}^{(1)} 
		&&=     \tiltheta_n^{(1)} 
		+ \alpha_{n+1} \bigl[ A  \tiltheta_n^{(1)}     +   \Delta_{n+2}^m  \bigr]   \,,    && \tiltheta_0^{(1)} = \tiltheta_0
		\label{e:clEa-1-app}
		\\
		&
		\tiltheta_{n+1}^{(2)} 
		&&=     \tiltheta_n^{(2)} 
		+ \alpha_{n+1} \bigl[   A  \tiltheta_n^{(2)}    -    \alpha_n [I+A]Z_{n+1}  \bigr]   \,, \qquad && \tiltheta_1^{(2)} = Z_1
		\label{e:clEa-2-app}%
		\end{alignat}%
		\label{e:decomp2-app}%
\end{subequations}
For fixed $\varrho < \varrho_0\text{ and } \varrho \leq \half$, Let $T>0$ solve the Lyapunov equation $[A+\varrho I]T+ T[A+\varrho I]^\transpose + I= 0$, which exists since $A+\varrho I $ is Hurwitz. Define the norm of $\tiltheta_n$ by $\|\tiltheta_n\|_T \eqdef \sqrt{\Expect[\tiltheta_n^\transpose T \tiltheta_n]}$.
	
	First consider $\tiltheta_n^{(1)}$. Since the martingale difference $\Delta_{n+2}^m$ is uncorrelated with $\tiltheta_n^{(1)}$, denoting $e_n = \|\tiltheta_n^{(1)}\|_T^2, b_{n+2} = \|\Delta_{n+2}^m\|_T^2$, we obtain the following from \eqref{e:clEa-1-app}:
	
\begin{equation}
	\label{e:mart-square}
	e_{n+1} = \|[I+\alpha_{n+1}A]\tiltheta_n^{(1)}\|_T^2 + b_{n+2} 
\end{equation}
	
	Letting $\lambda_\circ>0$ denote the largest eigenvalue of $T$, we arrive at the following simplification of the first term in \eqref{e:mart-square}
\begin{equation}
	\label{e:sq-simp-lya-mart}
	\begin{aligned}
	\|[I+\alpha_{n+1}A]\tiltheta_n^{(1)}\|_T^2 &
	= \Expect\bigl[(\tiltheta_n^{(1)})^\transpose[T -2\alpha_{n+1}\varrho T- \alpha_{n+1}I + \alpha_{n+1}^2 ATA^\transpose]\tiltheta_n^{(1)}\bigr] \\
	& \leq  \Expect\bigl[(\tiltheta_n^{(1)})^\transpose[T -2\alpha_{n+1}\varrho T-\frac{1}{\lambda_\circ}\alpha_{n+1}T + \alpha_{n+1}^2 ATA^\transpose]\tiltheta_n^{(1)}\bigr] \\
	&\leq [1-2\alpha_{n+1}\varrho-\alpha_{n+1}/\lambda_\circ + \alpha_{n+1}^2L^2] \|\tiltheta_n\|_T^2
	\end{aligned}
\end{equation}
where $L$ denotes the induced operator norm of $A$ with respect to the norm $\|\varble \|_T$.  \spm{yechh!  But I can't think of anything better.}
We then obtain the following recursive bound from \eqref{e:mart-square} and \eqref{e:sq-simp-lya-mart}
\[
	e_{n+1} \leq [1-(2\varrho+1/\lambda_\circ)\alpha_{n+1} +L^2\alpha_{n+1}^2] e_n + \alpha_{n+1}^2 K
\]
	where $K=\sup_{n\geq 1} b_n$. $K$ is finite since $b_n$ converges to $\Expect_\pi[(\Delta_n^m)^\transpose T \Delta_n^m]$ geometrically fast.
	
	Consequently, for each $n\geq 1$,
\[
	e_{n+1} \leq e_{0}\prod_{k=1}^{n+1}[1- (2\varrho+1/\lambda_\circ)\alpha_k +L^2\alpha_k^2]  + K \sum_{k=1}^{n+1} \alpha_k^2 \prod_{l=k+1}^{n+1} [1- (2\varrho+1/\lambda_\circ)\alpha_l +L^2\alpha_l^2] 
\]
	By \Lemma{t:prod-n},
\[
	\begin{aligned}
	e_{n+1}	 \leq e_{1}K_{\ref{t:prod-n}} \frac{1}{(n+2)^{2\varrho+1/\lambda_\circ}} + \frac{KK_{\ref{t:prod-n}} }{(n+2)^{2\varrho+1/\lambda_\circ}}\sum_{k=1}^{n+1} \alpha_k^{2-2\varrho - 1/\lambda_\circ}
	\end{aligned}
\]
	Therefore, $e_{n+1} \rightarrow 0$ at rate at least $n^{-2\varrho}$.
	
	For $\tiltheta_n^{(2)}$, we use similar arguments. We obtain the following from \eqref{e:clEa-2-app} by the triangle inequality.
\[
	\|\tiltheta_{n+1}^{(2)}\|_T \leq \|[I+\alpha_{n+1}A]\tiltheta_n^{(2)}\|_T + \alpha_n\alpha_{n+1}\|[I+A]Z_{n+1}\|_T
\]
Using the same argument as in \eqref{e:sq-simp-lya-mart}, along with the inequality $\sqrt{1+x} \leq 1 + \half x$,
\[
	\begin{aligned}
	\|[I+\alpha_{n+1}A]\tiltheta_n^{(2)}\|_T 
																& \leq \|\tiltheta_n^{(2)}\|_T\sqrt{1 -2\alpha_{n+1}\varrho- \alpha_{n+1}/\lambda_\circ + \alpha_{n+1}^2L^2} \\
																& \leq \|\tiltheta_n^{(2)}\|_T(1-\alpha_{n+1}\varrho - \alpha_{n+1}/(2\lambda_\circ) + \half\alpha_{n+1}^2L^2)
	\end{aligned}
\]
	 Denote $K'=\sup_{n\geq 1} \|[I+A]Z_{n+1}\|_T$.
\[
	\|\tiltheta_{n+1}^{(2)}\|_T \leq [1-(\varrho+ 1/(2\lambda_\circ))\alpha_{n+1}+\half\alpha_{n+1}^2L^2]\|\tiltheta_n^{(2)}\|_T  + \alpha_n\alpha_{n+1}K' 
\]
	Then by the same argument for the martingale difference term, we can show that $\|\tiltheta_n^{(2)}\|_T\rightarrow 0$ at rate at least $n^{-\varrho}$. 
	
	Given $\|\tiltheta_n^{(3)}\|_T=\alpha_n\|Z_{n+1}\|_T$ converges to zero at rate $1/n$, the proof is completed by the triangle inequality.
\end{proof}

\subsection{Proof of \Proposition{t:mart-3}}
\label{s:app-msr-msr-mart}

\paragraph*{(i)}

Recall that $\{ \Delta^m_n\}$ is a martingale difference sequence.
It is thus an uncorrelated sequence for which $\tiltheta_n^{(1)} $ and $\Delta^m_{n+k}$ are uncorrelated for $k\ge 2$. The following recursion is obtained from these facts and \eqref{e:clEa-1}
\[
\Cov(\theta_{n+1}^{(1)}) = \Cov(\theta_n^{(1)}) + \alpha_{n+1}\Bigl[\Cov(\theta_n^{(1)})A^\transpose + A\Cov(\theta_n^{(1)}) + \alpha_{n+1}[A\Cov(\theta_n^{(1)}) A^\transpose + \Sigma_{\Delta_{n+2}}]\Bigr]
\]
Multiplying each side by $n+1$ gives
\[
\begin{aligned}
(n+1)\Cov(\theta_{n+1}^{(1)})  =& n\Cov(\theta_n^{(1)}) + \Cov(\theta_n^{(1)}) + \Cov(\theta_n^{(1)})A^\transpose + A\Cov(\theta_n^{(1)}) \\
	&+ \alpha_{n+1}[A\Cov(\theta_n^{(1)}) A^\transpose + \Sigma_{\Delta_{n+2}}] \\
 = & n\Cov(\theta_n^{(1)}) + \alpha_{n+1} \Bigl[(1+\frac{1}{n})[n\Cov(\theta_n^{(1)}) + n\Cov(\theta_n^{(1)})A^\transpose + An\Cov(\theta_n^{(1)})] \\
 	&+ A\Cov(\theta_n^{(1)}) A^\transpose + \Sigma_{\Delta_{n+2}} \Bigr]
\end{aligned}
\]

The following argument will be used repeatedly through this Appendix:  the recursion for $n\Cov(\theta_n^{(1)}) $ is  a \textit{deterministic} SA recursion for $n\Cov(\theta_n^{(1)})$, and is regarded as an Euler approximation to the stable linear system
\begin{equation}
\ddt \clX(t)  =(1+e^{-t})[\clX(t) + A \clX(t) + \clX(t) A^\transpose] + \Sigma_{\Delta}  + e^{-t}A\clX(t)A^\transpose 
\label{e:SigmaODE}
\end{equation}
Stability follows from the assumption that $\half I+A$ is Hurwitz.
\spm{Note big changes}
The standard justification of the Euler approximation is through the choice of timescale: let $t_n=\sum_{k=1}^n\alpha_k$  and    let $\clX^n(t)$ denote the solution to this ODE  on $[t_n,\infty)$ with $\clX^n(t_n) = n\Cov(\theta_n^{(1)}) $,  $t\ge t_n$, for any $n\ge 1$.  Using standard ODE arguments \cite{bor08a}, 
\[
\sup_{k\ge n}  \|  \clX^n(t_k) -  k \Sigma_k^{(1)} \|  = O(1/n)
\]
Exponential convergence of $\clX $ to $\Sigma_\theta$ implies convergence of $\{n\Cov(\theta_n^{(1)})\}$ to zero at rate 
$1/n^\delta$ for some $\delta = \delta(\half I+A, \Sigma_{\Delta})>0$.  
\qed

\paragraph*{(ii)}
Denote $e_n^{\varrho_0} = \Expect[|v^\transpose\tiltheta_n^{\varrho_0}|^2]$ and $\lambda = -\varrho_0 + ui$. We begin with the proof that
\begin{equation}
\liminf_{n\to\infty}  e^{\varrho_0}_n > 0
\label{e:liminfe}
\end{equation}

With $v^\transpose [I\lambda - A] =0$, we have
$v^\transpose [I\varrho_n + A(n,\varrho)] =  [\epsy_v(n,\varrho_0) + u i] v^\transpose $,   with $\epsy_v(n,\varrho_0) = O(n^{-1})$.   
Applying  \eqref{e:clE-1} gives 
\[
v^\transpose   \tiltheta_{n+1}^{{\varrho_0},(1)}  
=     v^\transpose   \tiltheta_n^{{\varrho_0},(1)} 
+ \alpha_{n+1} \bigl[ [\epsy_v(n,{\varrho_0}) +u i]  v^\transpose \tiltheta_n^{{\varrho_0},(1)}     +  (n+1)^{\varrho_0}   v^\transpose  \Delta_{n+2}^m  \bigr]   
\]
Let $\overline{v}$ denote the conjugate of $v$. Consequently, with  $\sigma^2_n(v) =  v^\transpose \Sigma_{\Delta_n}  \overline{v}$,  
\[ 
e^{\varrho_0}_{n+1}  =    \bigl[[1+  \epsy_v(n,{\varrho_0}) /(n+1)]^2 + u^2/(n+1)^2 \bigr] e^{\varrho_0}_n     +   (n+1)^{2{\varrho_0} - 2}   \sigma^2_{n+2}(v) 
\]
$V$-uniform ergodicity implies that  $\sigma^2_n(v) \to v^\transpose \Sigma_{\Delta}  \overline{v} >0$ as $n\to\infty$ at a geometric rate.   
Fix $n_0>0$ so that $\sigma^2_{n_0}(v)>0$,  and hence also $e^{\varrho_0}_{n_0+1}>0$.    
We also assume that  $1+  \epsy_v(n,{\varrho_0}) /(n+1)>0$ for   $n\ge n_0$,  which is possible since   $ \epsy_v(n,{\varrho_0} )=O(n^{-1})$.

For $N > n_0$ we obtain the uniform bound
\[
\log(e^{\varrho_0}_N )  \ge \log(e^{\varrho_0}_{n_0+1})  + 2 \sum_{n=n_0+2}^\infty \log  [1 - | \epsy_v(n,{\varrho_0}) |/(n+1)]  >-\infty
\]
which proves that $ \liminf_{n\to\infty} e_n^{\varrho_0} =\liminf_{n\to\infty}  v^\transpose \Sigma_n^{{\varrho_0},(1)} \overline{v} >0$.

The proof of an upper bound for ${\varrho_0}<1/2$:   by concavity of the logarithm,
\[ 
\log(e^{\varrho_0}_{n+1} )  \le \log\bigl( \bigl[ [1+  \epsy_v(n,{\varrho_0}) /(n+1)]^2 + u^2/(n+1)^2\bigr]
e^{\varrho_0}_n   \bigr)   
+ 
K   (n+1)^{2{\varrho_0} - 2}     
\]
where $K = \sup_{n >n_0}     \bigl[ [1+  \epsy_v(n,{\varrho_0}) /(n+1)]^2 + u^2/(n+1)^2 \bigr]^{-1} [e^{\varrho_0}_n]^{-1}
\sigma^2_{n+2}(v) $.    Using concavity of the logarithm once more gives  
\[ 
\log(e^{\varrho_0}_{n+1} )  \le   \log(e^{\varrho_0}_n )   + 2  \epsy_v(n,{\varrho_0}) /(n+1) +\frac{\epsy_v(n,{\varrho_0})^2}{(n+1)^2} + \frac{u^2}{(n+1)^2}
+ 
K   (n+1)^{2{\varrho_0} - 2}   
\]
which gives the uniform upper bound
\[
\log(e^{\varrho_0}_N )  \le  \log(e^{\varrho_0}_{n_0+1})  +   \sum_{n=n_0+2}^\infty  \Bigl(2 \frac{ | \epsy_v(n,{\varrho_0}) | }{n+1) }  +\frac{\epsy_v(n,{\varrho_0})^2}{(n+1)^2} + \frac{u^2}{(n+1)^2} + K   (n+1)^{2{\varrho_0} - 2}   \Bigr)  <\infty
\]
This proves that $ \limsup_{n\to\infty} e_n^{\varrho_0}=\limsup_{n\to\infty}v^\transpose \Sigma_n^{{\varrho_0},(1)} \overline{v} < \infty$.
\qed

\subsection{Proof for \Proposition{t:tele-3}}
\label{s:app-msr-tele}

\paragraph*{(i)}

Denote $\clD_n = \epsy(n,\varrho)I + A(n, \varrho) - A$.    
We can rewrite \eqref{e:clE-2} as
\begin{equation}
\label{e:clE-2-var}
\begin{aligned}
\tiltheta_{n+1}^{\varrho,(2)} 
&=     \tiltheta_n^{\varrho,(2)} 
+ \alpha_{n+1} \bigl[   [  \half I + A]  \tiltheta_n^{\varrho,(2)} + \clD_n\tiltheta_n^{\varrho,(2)}    -  \alpha_n(n+1)^\varrho  [I+A]    Z_{n+1}
\bigr]  \\
&= \bigl[I
+ \alpha_{n+1} [  \half I + A]\bigr]\tiltheta_n^{\varrho,(2)} + \alpha_{n+1} \clD_n\tiltheta_n^{\varrho,(2)}    -  \alpha_{n+1}\alpha_n(n+1)^\varrho  [I+A]    Z_{n+1}
\end{aligned}
\end{equation}

Let $T > 0$ solve the Lyapunov equation
\[
[\half I + A]^\intercal T + T[\half I + A ] + I = 0
\]
As in the proof of \Lemma{t:converge-sa}, a solution exists because  $\half I + A$ is Hurwitz.   Adopting the familiar notation   $\|\tiltheta_n^{\varrho,(2)}\|_T \eqdef \sqrt{\Expect[(\tiltheta_n^{\varrho,(2)})^\intercal T\tiltheta_n^{\varrho,(2)}]}$,   the triangle inequality applied to \eqref{e:clE-2-var} gives
\begin{equation}
\label{e:clE-2-lya}
\|\tiltheta_{n+1}^{\varrho,(2)}\|_T
\leq \|\bigl[I+ \alpha_{n+1} [  \half I + A]\bigr]\tiltheta_n^{\varrho,(2)}\|_T
+ \alpha_{n+1}\|\clD_n\|_T\|\tiltheta_n^{\varrho,(2)}\|_T + \alpha_{n+1}\alpha_n (n+1)^{\varrho}\|[I+A]Z_{n+1}\|_T
\end{equation}
The first term can be simplified by the Lyapunov equation.
\[
\begin{aligned}
\|\bigl[I+ \alpha_{n+1} [  \half I + A]\bigr]\tiltheta_n^{\varrho,(2)}\|_T^2 =& \Expect\bigl[(\tiltheta_n^{\varrho,(2)})^\intercal\bigl[ T - \alpha_{n+1} I + \alpha_{n+1}^2 [  \half I + A]^\intercal T[  \half I + A] \bigr]\tiltheta_n^{\varrho,(2)}\bigr] \\
\leq &\Expect\bigl[(\tiltheta_n^{\varrho,(2)})^\intercal\bigl[ T - \frac{\alpha_{n+1}}{\lambda_\circ} T + \alpha_{n+1}^2 [  \half I + A]^\intercal T[  \half I + A] \bigr]\tiltheta_n^{\varrho,(2)}\bigr] \\
\leq & \|\tiltheta_n^{\varrho,(2)}\|_T^2-\frac{\alpha_{n+1}}{\lambda_\circ}\|\tiltheta_n^{\varrho,(2)}\|_T^2+ \alpha_{n+1}^2 L^2 \|\tiltheta_n^{\varrho,(2)}\|_T^2
\end{aligned}
\]
where $L$ is the induced operator norm of $\half I + A$, and $\lambda_\circ>0$ denotes its largest eigenvalue.

Consequently, by the inequality $\sqrt{1+x} \leq 1 + \half x$,
\[
\begin{aligned}
\|\bigl[I+ \alpha_{n+1} [  \half I + A]\bigr]\tiltheta_n^{\varrho,(2)}\|_T \leq \|\tiltheta_n^{\varrho,(2)}\|_T\sqrt{1-\frac{\alpha_{n+1}}{\lambda_\circ} + \alpha_{n+1}^2 L^2} \leq \|\tiltheta_n^{\varrho,(2)}\|_T(1-\frac{\alpha_{n+1}}{2\lambda_\circ}+\half \alpha_{n+1}^2 L^2)
\end{aligned}
\]
Fix $n_0>0$ such that for $n\geq n_0$,
\[
\begin{aligned}
1-\frac{\alpha_{n+1}}{2\lambda_\circ}+\half \alpha_{n+1}^2 L^2+\alpha_{n+1}\|\clD_n\|_T \leq  1-\frac{\alpha_{n+1}}{4\lambda_\circ}
\end{aligned}
\]
This is possible since $\|\clD_n\|_T=O(n^{-1})$. 

Denote $\delta = \min(\frac{1}{4\lambda_\circ},\frac{1}{4})$ and $K = \sup_{n\geq n_0}\|[I+A]Z_{n+1}\|_T$, which is finite because $\|Z_{n+1}\|_T$ converges. We obtain the following from \eqref{e:clE-2-lya}
\begin{equation}
\label{e:cle-2-inequ}
\begin{aligned}
\|\tiltheta_{n+1}^{\varrho,(2)}\|_T\leq  & \|\tiltheta_n^{\varrho,(2)}\|_T(1-\delta\alpha_{n+1})+ \alpha_{n+1}^{1/2}\alpha_n K  \\
\leq & \|\tiltheta_n^{\varrho,(2)}\|_T(1-\delta\alpha_{n+1})+ \alpha_n^{3/2} K 
\end{aligned}
\end{equation}
Apply \eqref{e:cle-2-inequ} repeatedly for $n\geq n_0$
\[
\begin{aligned}
\|\tiltheta_{n+1}^{\varrho,(2)}\|_T
\leq & \|\tiltheta_{n_0}^{\varrho,(2)}\|_T\prod_{k=n_0+1}^{n+1} (1-\delta\alpha_k)  + K \sum_{k=n_0}^{n}  \alpha_k^{3/2}\prod_{l=k+1}^{n} (1-\delta \alpha_l) \\
\leq &  \|\tiltheta_{n_0}^{\varrho,(2)}\|_T\exp(\delta)\frac{n_0^\delta}{(n+2)^\delta}  + \frac{K\exp(\delta)}{(n+1)^{\delta}} \sum_{k=n_0}^{n} k^{-\frac{3}{2} + \delta}
\end{aligned}
\]
where $\sum_{k=1}^\infty k^{-\frac{3}{2}+\delta} <\infty$ for  $\delta \leq 1/4$. Therefore,  $\|\tiltheta_n^{\varrho,(2)}\|_T\rightarrow 0$ at rate at least $n^{-\delta}$. 

The desired conclusion follows:  
letting $\lambda_\bullet>0$ denote the smallest eigenvalue of $T$,
\[
\Sigma_n^{\varrho,(2)} \leq \Expect[(\tiltheta_n^{\varrho,(2)})^\intercal \tiltheta_n^{\varrho,(2)}] I \leq \frac{1}{\lambda_\bullet}\|\tiltheta_n^{\varrho,(2)}\|_T^2I
\]
\qed

\paragraph*{(ii)}
Multiplying both sides of \eqref{e:clE-2} by $v^\transpose$ gives
\begin{equation}
\label{e:clE-2-v}
\begin{aligned}
v^\transpose \tiltheta_{n+1}^{\varrho_0, (2)} &= v^\transpose \tiltheta_n^{\varrho_0, (2)} + \alpha_{n+1}\bigl[[\epsy_v(n,\varrho_0) + ui]v^\transpose\tiltheta_n^{\varrho_0,(2)} - (1-\varrho_0 +ui)\alpha_n(n+1)^{\varrho_0} v^\transpose Z_{n+1}\bigr] \\
& = \bigl[1+ \alpha_{n+1}[\epsy_v(n,\varrho_0) + ui]\bigr]v^\transpose \tiltheta_n^{\varrho_0, (2)} - (1-\varrho_0 + ui)\alpha_n\alpha_{n+1} (n+1)^{\varrho_0} v^\transpose Z_{n+1}
\end{aligned}
\end{equation}
With $\|v^\transpose \tiltheta_n^{\varrho_0,(2)}\|_2  \eqdef \sqrt{\Expect[|v^\transpose \tiltheta_n^{\varrho_0,(2)}|^2]}$, we obtain the following from \eqref{e:clE-2-v} by the triangle inequality
\begin{equation}
\label{e:cle-2-e}
\begin{aligned}
\|v^\transpose \tiltheta_{n+1}^{\varrho_0,(2)}\|_2
	\leq \bigl|1+ \alpha_{n+1}[\epsy_v(n,\varrho_0) + ui]\bigr| \|v^\transpose \tiltheta_n^{\varrho_0,(2)}\|_2
		+  \bigl|1-\varrho_0 + ui \bigr|\alpha_n\alpha_{n+1} (n+1)^{\varrho_0} \|v^\transpose Z_{n+1}\|_2 
\end{aligned}
\end{equation}	
By the inequality $\sqrt{1+x} \leq 1 + \half x$, we have
\[
\bigl|1+ \alpha_{n+1}\epsy_v(n,\varrho_0) + \alpha_{n+1}u i \bigr| \leq 1+ \alpha_{n+1}\epsy_v(n,\varrho_0) + \half \alpha_{n+1}^2\epsy_v(n,\varrho_0) ^2+ \half \alpha_{n+1}^2u^2
\]
Fix $n_0>0$ such that for $n\geq n_0$,
\[
1+ \alpha_{n+1}\epsy_v(n,\varrho_0) + \half \alpha_{n+1}^2\epsy_v(n,\varrho_0) ^2+ \half \alpha_{n+1}^2u^2\leq 1+\alpha_{n+1}^{3/2}
\]
which is possible since $\epsy_v(n,\varrho_0) = O(n^{-1})$. 
With $K=\sup_{n\geq n_0} |1-\varrho_0 + ui|  \| v^\transpose Z_{n+1}\|_2$, we obtain the following bound from \eqref{e:cle-2-e}:
\begin{equation}
\label{e:clE-2-rho0}
\|v^\transpose \tiltheta_{n+1}^{\varrho_0,(2)}\|_2\leq (1+\alpha_{n+1}^{3/2})\|v^\transpose \tiltheta_n^{\varrho_0,(2)}\|_2+  \alpha_n^{2-\varrho_0} K
\end{equation}
Iterating \eqref{e:clE-2-rho0} gives,
\[
\begin{aligned}
\|v^\transpose \tiltheta_{n+1}^{\varrho_0,(2)}\|_2&\leq \|v^\transpose \tiltheta_{n_0}^{\varrho_0,(2)}\|_2\prod_{k=n_0+1}^{n+1} (1+ \alpha_{k}^{3/2}) + K\sum_{k=n_0}^{n}\alpha_{k}^{2-\varrho_0}\prod_{l=k+1}^n (1+\alpha_l^{3/2}) \\
&\leq \|v^\transpose \tiltheta_{n_0}^{\varrho_0,(2)}\|_2\exp(\sum_{k=n_0+1}^{n+1} k^{-2/3})  + K\sum_{k=n_0}^{n} k^{-2+\varrho_0} \exp(\sum_{l=k+1}^{n} l^{-3/2}) 
\end{aligned}
\]
$\limsup_{n\rightarrow \infty} \| v^\transpose \tiltheta_n^{\varrho_0,(2)})\|_2 <\infty$, since it is assumed that $\varrho_0 <\half$.
\qed

\subsection{Proof of \Proposition{t:fine-comp}}
\label{s:app-finer}

\paragraph*{(i)}

Since $\Delta_{n+2}^m$ is uncorrelated with $\tiltheta_n^{(1)}$, the following recursion follows from \eqref{e:clEa-1}:
\[
\Cov(\theta_{n+1}^{(1)}) = \Cov(\theta_n^{(1)}) + \alpha_{n+1}\Bigl[\Cov(\theta_n^{(1)})A^\transpose + A\Cov(\theta_n^{(1)}) + \alpha_{n+1}[A\Cov(\theta_n^{(1)}) A^\transpose + \Sigma_{\Delta_{n+2}}]\Bigr]
\]
Take $\varrho=1/2$ in the definition of $\tiltheta^{\varrho,(1)}$ and $\Sigma_n^{\varrho,(1)} =\Expect[\tiltheta^{\varrho,(1)}(\tiltheta^{\varrho,(1)})^\transpose] = n\Cov(\theta_n^{(1)})$. Multiplying each side of the equation by $n+1$ gives
\begin{equation}
\label{e:mart-sigma}
\Sigma_{n+1}^{\varrho,(1)} = \Sigma_n^{\varrho,(1)} + \alpha_{n+1}\Bigl[(1+\frac{1}{n})\bigl[\Sigma_n^{\varrho,(1)}+\Sigma_n^{\varrho,(1)}A^\transpose + A\Sigma_n^{\varrho,(1)}\bigr] + \frac{1}{n}A\Sigma_n^{\varrho,(1)} A^\transpose + \Sigma_{\Delta_{n+2}}  \Bigr] 
\end{equation}

Recall that $\Sigma_\theta$ solves the Laypunov equation $\Sigma+ \Sigma A^\transpose +A\Sigma + \Sigma_\Delta= 0$. Denoting $E_n=\Sigma_n^{\varrho,(1)} - \Sigma_\theta$, the following identity holds
\[
\Sigma_n^{\varrho,(1)}+\Sigma_n^{\varrho,(1)}A^\transpose + A\Sigma_n^{\varrho,(1)} = E_n + E_n A^\transpose + A E_n-\Sigma_\Delta
\]
Subtracting $\Sigma_\theta$ from both sides of \eqref{e:mart-sigma} gives the recursion
\begin{equation}
\label{e:mart-sigma-bar}
\begin{aligned}
E_{n+1}= E_n + \alpha_{n+1}\Bigl[&(1+\frac{1}{n})\bigl[E_n + E_nA^\transpose  + AE_n\bigr]
+\frac{1}{n}AE_n A^\transpose \\
&+\frac{1}{n}A\Sigma_\theta A^\transpose - \frac{1}{n}\Sigma_{\Delta} -\Sigma_{\Delta} +\Sigma_{\Delta_{n+2}} \Bigr]
\end{aligned}
\end{equation}
Similar to the decomposition in \eqref{e:decomp2}, we have $E_n= E_n^{(1)} + E_n^{(2)}$, each evolving as
\begin{subequations}
	\begin{alignat}{1}
	E_{n+1}^{(1)}
	&=E_n^{(1)} + \alpha_{n+1}\Bigl[(1+\frac{1}{n})\bigl[E_n^{(1)}+E_n^{(1)} A^\transpose + AE_n^{(1)}\bigr] + \frac{1}{n}AE_n^{(1)} A^\transpose +\frac{1}{n}\bigl[A\Sigma_\theta A^\transpose - \Sigma_{\Delta}\bigr]\Bigr] 
	\label{e:mart-sigma-bar-a}
	\\
	E_{n+1}^{(2)}
	&= E_n^{(2)} + \alpha_{n+1}\Bigl[(1+\frac{1}{n})\bigl[E_n^{(2)}+ E_n^{(2)}A^\transpose + AE_n^{(2)}\bigr] + \frac{1}{n}A E_n^{(2)} A^\transpose + \Sigma_{\Delta_{n+2}}-\Sigma_{\Delta} \Bigr] 
	\label{e:mart-sigma-bar-b}
	\end{alignat}%
	\label{e:mart-sigma-decomp}%
\end{subequations}%
Since $\Sigma_{\Delta_{n+2}} - \Sigma_{\Delta}$ converges to zero geometrically fast, $\{E_n^{(1)}\}$ converges to zero faster than $\{E_n^{(2)}\}$. 
\spm{\bl{SPM to elaborate on ``converges faster".} }

Multiplying each side of \eqref{e:mart-sigma-bar-a} by $n+1$ gives
\[
\begin{aligned}
(n+1)E_{n+1}^{(1)}
&= (n+1)E_n^{(1)}+ (1+\frac{1}{n})\bigl[E_n^{(1)}+E_n^{(1)} A^\transpose + AE_n^{(1)}\bigr] + \frac{1}{n}\bigl[AE_n^{(1)} A^\transpose +A\Sigma_\theta A^\transpose - \Sigma_{\Delta}\bigr] \\
& = nE_n^{(1)} +\frac{1}{n}\Bigl[(1+\frac{1}{n})\bigl[2nE_n^{(1)}+nE_n^{(1)} A^\transpose + AnE_n^{(1)}\bigr] +A\Sigma_\theta A^\transpose - \Sigma_{\Delta}+\clE_n^{\bullet,(1)}  \Bigr]
\end{aligned}
\]
with the error term $\clE_n^{\bullet,(1)} = AE_n^{(1)} A^\transpose - E_n$. Note that $A\Sigma_\theta A^\transpose - \Sigma_{\Delta} =[A+I]\Sigma_\theta[A+I]^\transpose $ is positive definite.

The recursion for $\{nE_n^{(1)} \}$ is treated as in  the proof of   \Proposition{t:mart-3}~(i).   Consider the matrix ODE,
\begin{equation}
\ddt \clX(t)  =(1+e^{-t})[2\clX(t) + \clX(t) A^\transpose + A \clX(t)] + A\Sigma_\theta A^\transpose - \Sigma_{\Delta} + e^{-t}[A\clX(t)A^\transpose -\clX(t)]
\label{e:Sigma-ODE-a}
\end{equation}
Let $t_n=\sum_{k=1}^n 1/k$  and    let $\clX^n(t)$ denote the solution to this ODE  on $[t_n,\infty)$ with $\clX^n(t_n) = nE_n^{(1)} $,  $t\ge t_n$, for any $n\ge 1$.  We then obtain as previously,     \spm{You can't repeat phrases like this --- we really should have a lemma, or quote a reference: Using standard ODE arguments we obtain}
\[
\sup_{k\ge n}  \|  \clX^n(t_k) -   kE_k^{(1)} \|  = O(1/n)
\]
Recall that $\Sigma_{\sharp}^{(1)} \ge 0$ is the solution to the Lyapunov equation \eqref{e:SigLya-2}. Exponential convergence of $\clX$ to $\Sigma_{\sharp}^{(1)}$ implies convergence of $\{nE_n^{(1)}\}$ at rate $1/n^\delta$
for $\delta=\delta(A+I, \Sigma_{\Delta})>0$.  Therefore, $nE_n = \Sigma_{\sharp}^{(1)} + O(n^{-\delta})$.

Given $\Cov(\theta_n^{(1)}) = n^{-1}\Sigma_\theta + n^{-1}E_n$, we have
\[
\Cov(\theta_n^{(1)}) = n^{-1}\Sigma_\theta + n^{-2} \Sigma_{\sharp}^{(1)} + O(n^{-2-\delta})
\]
\qed

\paragraph*{(ii)}

We focus on $R_n^{(2,1),(1)}$ since $R_n^{(1), (2,1)} = [R_n^{(2,1),(1)}]^\transpose$. Recall the update forms of $\tiltheta_n^{(1)}$ and $\tiltheta_n^{(2,1)}$ in \eqref{e:clEa-1} and \eqref{e:clE-finer-1} respectively, where $\tiltheta_n^{(1)}$ is uncorrelated with the martingale difference sequence $\{\haDelta_{n+k}^m\}$ for $k\ge 2$ and $\tiltheta_n^{(2,1)}$ is uncorrelated with $\{\Delta_{n+k}^m\}$ for $k\ge 2$. With $R_n^{(2,1),(1)} = \Expect[\tiltheta_n^{(2,1)}(\tiltheta_n^{(1)})^\transpose]$, the following is obtained from these facts:
\[
\begin{aligned}
R_{n+1}^
{(2,1),(1)} =R_n^{(2,1),(1)} + \alpha_{n+1} \bigl[R_n^{(2,1),(1)}A^\transpose &+ AR_n^{(2,1),(1)} + \alpha_{n+1}AR_n^{(2,1),(1)}A^\transpose\\
& - \alpha_n\alpha_{n+1} [I+A]\Cov(\haDelta_{n+2}^m,\Delta_{n+2}^m)\bigr]
\end{aligned}
\]
Denote $C_n = nR_n^{(2,1),(1)}$.
Multiplying both sides of the previous equation by $n+1$ gives
\[
C_{n+1} = C_n +\alpha_{n+1}\bigl[(1+n^{-1})[C_n+C_nA^\transpose + AC_n] + \alpha_nAC_nA^\transpose - \alpha_n[I+A]\Cov(\haDelta_{n+2}^m,\Delta_{n+2}^m) \bigr] 
\]
Multiplying each side of this equation by $n+1$ once more results in
\[
\begin{aligned}
(n+1)C_{n+1} &= (n+1)C_n +(1+n^{-1})[C_n+C_nA^\transpose + AC_n] + \alpha_nAC_nA^\transpose - \alpha_n[I+A] \Cov(\haDelta_{n+2}^m,\Delta_{n+2}^m) \\
&= nC_n + \alpha_n\bigl[ (1+n^{-1})[2nC_n+nC_nA^\transpose + AnC_n] - [I+A] \Cov_\pi(\haDelta_{n+2}^m,\Delta_{n+2}^m) + \clD_{n+1}^{(2)} \bigr]
\end{aligned}
\]
where the error term $\clD_{n+1}^{(2)}$ consists of two components: $[I+A][\Cov_\pi(\haDelta_{n+2}^m,\Delta_{n+2}^m)-\Cov(\haDelta_{n+2}^m,\Delta_{n+2}^m)]$ that converges to zero at a geometric rate and $AC_nA^\transpose - C_n$.

As previously, this is approximated by the linear system
\begin{equation}
\begin{aligned}
\ddt \clX(t)  =&(1+e^{-t})[2\clX(t) + \clX(t) A^\transpose + A \clX(t)] + e^{-t}[A\clX(t)A^\transpose -\clX(t)] \\
&  \qquad -[I+A] \Cov_\pi(\haDelta_{n+2}^m,\Delta_{n+2}^m)) 
\end{aligned}
\label{e:Sigma-ODE-cov}
\end{equation}
With the same argument used in (i), $\{nC_n + nC_n^\transpose\}$ converges to $\Sigma_{\sharp}^{(2)}$ in \eqref{e:SigLya-c} at rate $1/n^\delta$ for $\delta=\delta(A+I)>0$. Therefore, $nC_n + nC_n^\transpose = \Sigma_{\sharp}^{(2)} + O(n^{-\delta})$ and $ R_n^{(2,1),(1)}= n^{-2}C_n = n^{-2}\Sigma_{\infty,C} + O(n^{-2-\delta})$.
\qed

\paragraph*{(iii)}

The third claim in \Proposition{t:fine-comp} is established through a sequence of lemmas. Start with the representation of $\tiltheta_{n+1}^{(3)}$ based on \eqref{e:DoubleFish-2}:
\[
\tiltheta_{n+1}^{(3)} = -\frac{1}{n+1}Z_{n+2} = -\frac{1}{n+1}\haDelta_{n+3}^m + \frac{1}{n+1}(\haZ_{n+3}-\haZ_{n+2})
\]
Since $\haDelta_{n+3}^m$ is uncorrelated with the sequence $\{\tiltheta_k^{(1)}\}$ for $k\le n+1$, we have
\begin{equation}
\Expect[\tiltheta_{n+1}^{(1)}(\haDelta_{n+3}^m)^\transpose]= 0
\label{e:cov-theta-mart}
\end{equation}
Hence it suffices to consider the correlation between $\tiltheta_{n+1}^{(1)}$ and $\haZ_{n+3}- \haZ_{n+2}$. The formula for $\tiltheta_{n+1}^{(1)}$ for $n\ge 1$ is
\begin{equation}
\tiltheta_{n+1}^{(1)} = \prod_{k=1}^{n+1}[I+\alpha_{k}A] \tiltheta_0 + \sum_{k=1}^{n+1}\alpha_k\prod_{l=k+1}^{n+1}[I+\alpha_lA] \Delta_{k+1}^m
\label{e:linear-solution-mart}
\end{equation}
$ \tiltheta_0\Expect[\haZ_{n+3}^\transpose-\haZ_{n+2}^\transpose]$ converges to zero geometrically fast under $V$-uniform ergodicity  of $\bfPhi$. Then we consider the expectation of the following:
\begin{equation}
\begin{aligned}
& \sum_{k=1}^{n+1} \alpha_k \prod_{l=k+1}^{n+1}[I+\alpha_lA]\Delta_{k+1}^m[\haZ_{n+3}^\transpose -\haZ_{n+2}^\transpose ]  \\
=& \sum_{k=1}^{n+1} \alpha_k \prod_{l=k+1}^{n+1}[I+\alpha_lA]  \bigl[\Delta_{k+2}^m\haZ_{n+3}^\transpose -  \Delta_{k+1}^m\haZ_{n+2}^\transpose \bigr] + \sum_{k=1}^{n+1}\alpha_k \prod_{l=k+1}^{n+1}[I+\alpha_lA] \bigl[\Delta_{k+1}^m - \Delta_{k+2}^m \bigr]\haZ_{n+3}^\transpose
\end{aligned}
\label{e:cov-mart-Z}
\end{equation}

The definition of $T$ is now based on the assumption that  $I+A$ is Hurwitz:
 $T > 0$ is the unique solution to the Lyapunov equation:
\[
[A+I]T + T[A+I]^\transpose +I = 0
\]
As previously, we denote $\|W\|_T^2 =  \Expect[W^\transpose T W]$ for a random vector $W$,  and denote by $\|M\|_T$ the induced operator norm of a matrix $M\in \Re^{d\times d}$. In the following result the vector $W$ is taken to be deterministic. 
\spm{watch for $\|\cdot\|_T$.}
\begin{lemma}
\label{t:prod-matrix}
Suppose the matrix $I+A$ is Hurwitz. Then there exists constant $K$ such that the following holds for any $k\ge 1$ and all $n \ge k$
\[
\bigl\| \prod_{l=k+1}^{n+1} [I+\alpha_l A] \bigr\|_T \leq K\frac{k}{n+2}
\]
\end{lemma}
\begin{proof}
For any vector $W\in\Re^d$ and $l\geq 1$, we have
\[
\begin{aligned}
\|[I+\alpha_lA]W\|_T^2 &= W^\transpose[T-2\alpha_l T -\alpha_l I +\alpha_l^2A^\transpose T A]W  \\
&\leq  W^\transpose[TT-2\alpha_l T +\alpha_l^2A^\transpose T A]W \\
&\leq (1-2\alpha_l + \alpha_l^2 L^2) \|W\|_T^2
\end{aligned}
\]
where $L=\| A\|_T$.   \spm{please check -- previously the statement was $L=\| A\|_T$.  I think it should be squared, and the error may appear elsewhere.}
Hence
\[
\|I+\alpha_l A\|_T \leq \sqrt{1-2\alpha_l + \alpha_l^2 L^2} \leq 1-\alpha_l + \half \alpha_l^2 L^2
\]
\Lemma{t:prod-n} completes the proof: 
\[
\begin{aligned}
\bigl\| \prod_{l=k+1}^{n+1} [I+\alpha_l A] \bigr\|_T 
\le 
\prod_{l=k+1}^{n+1} \|[I+\alpha_l A]\|_T \leq \prod_{l=k+1}^{n+1}[1-\alpha_l + \half L^2\alpha_l^2] \leq K_{\ref{t:prod-n}} \frac{k}{n+2}
\end{aligned}
\]
%Since linear operator norm for square matrices is sub-multiplicative, the proof is then completed.
\end{proof}

To analyze $\Expect[\Delta_{k+2}^m\haZ_{n+3}^\transpose] $, consider the bivariate Markov chain $\Phi_n^*=(\Phi_n, \Phi_{n+1})$,  $n\ge 0$, with state space $\zstate^*=\zstate\times\zstate$. 
An associated \textit{weighting function} $V^*:\zstate\times \zstate \rightarrow [1,\infty)$ is defined as $V^*(z,z') = V(z) + V(z')$.

Denote function $h^{k+1,n+2}:\zstate^* \rightarrow \Re^{d\times d}$ as $h^{k+1,n+2}(z', z'')=(\haf(z'') - \Expect[\haf(\Phi_{k+1})\mid \Phi_k =z'])\Expect[\hahaf(\Phi_{n+2})^\transpose \mid \Phi_{k+1}=z'']$ and $h_{i,j}^{k+1,n+2}:\zstate^* \rightarrow\Re$ as the $(i,j)$-th entry of $h^{k+1,n+2}$ for $1\le i,j \le d$. Note that $h^{k+1,n+2}(\Phi_k, \Phi_{k+1})=\Expect[\Delta_{k+1}^m \haZ_{n+2}\mid \clF_{k+1}]$

\begin{lemma}
\label{t:bivariate-chain}
Suppose Assumptions (A1) and (A3) hold. For each $1\le i,j \le d$, 
\begin{romannum}
\item $h_{i,j}^{k+1,n+2} \in \LVstar$, moreover there exists constant $B$ such that
\[
\|h_{i,j}^{k+1,n+2}\|_{V^*} \leq B\|\haf_i\|_{\sqrtV} \|\hahaf_j\|_{\sqrtV} \rho^{n-k+1}
\]
\item Consequently, there exists constant $B'$ such that
\[
\bigl|\Expect[h_{i,j}^{k+1,n+2}(\Phi_k, \Phi_{k+1})\mid \Phi_0=z] - \pi\bigl(h_{i,j}^{k+1,n+2}\bigr)\bigr| \leq B'\|\haf_i\|_{\sqrtV} \|\hahaf_j\|_{\sqrtV} V(z) \rho^{n+1}
\]
\end{romannum}
\end{lemma}

\begin{proof}
By the definition of $V^*$-norm,
\[
\begin{aligned}
\|h_{i,j}^{k+1,n+2}\|_{V^*} 
										= &  \sup_{z',z'' \in \zstate} \frac{\bigl|\bigl[\haf_i(z'') + \Expect[\haf_i(\Phi_{k+1})\mid \Phi_k = z']\bigr] \Expect[\hahaf_j(\Phi_{n+2})\mid \Phi_{k+1}=z'']\bigr|}{V(z') + V(z'')}    \\
										\leq & \sup_{z'' \in \zstate} \frac{\bigl|\haf_i(z'') \Expect[\hahaf_j(\Phi_{n+2}) \mid \Phi_{k+1}=z'']\bigr|}{V(z'')} \\
										&+ \sup_{z',z''\in\zstate} \frac{\bigl| \Expect[\haf_i(\Phi_{k+1})\mid \Phi_k = z']  \Expect[\hahaf_j(\Phi_{n+2}) \mid \Phi_{k+1}=z''] \bigr|}{V(z') + V(z'')}
\end{aligned}
\]
Given $\hahaf_j^2\in\LV$ and the $\sqrt{V}$-uniform ergodicity of $\bfPhi$~\cite[Lemma 15.2.9]{MT},
 there exists constant $B_{\sqrtV}< \infty$ such that
\[
\bigl| \Expect[\hahaf_j(\Phi_{n+2})\mid \Phi_{k+1}=z'']\bigr| \leq B_{\sqrtV} \|\hahaf_j\|_{\sqrtV} \sqrt{V(z'')} \rho^{n+1-k}
\]
Consequently,
\begin{equation}
\sup_{z'' \in \zstate} \frac{|\haf_i(z'') \bigl[ \Expect[\hahaf_j(\Phi_{n+2})\mid \Phi_{k+1}=z'']|}{V(z'')} \leq \|\haf_i\|_{\sqrtV}B_{\sqrtV} \|\hahaf_j\|_{\sqrtV} \rho^{n+1-k}
\label{e:Vstar-1}
\end{equation}
By the inequality $V(z') +V(z'') \geq \sqrt{V(z')V(z'')}$ and the $\sqrt{V}$-uniform ergodicity of $\bfPhi$ once more, we have
\begin{equation}
\begin{aligned}
&\sup_{z',z''\in\zstate} \frac{\bigl| \Expect[\haf_i(\Phi_{k+1})\mid \Phi_k = z']  \Expect[\hahaf_j(\Phi_{n+2})\mid \Phi_{k+1}=z''] \bigr|}{V(z') + V(z'')} \\
\leq & \sup_{z'\in\zstate}\frac{\bigl| \Expect[\haf_i(\Phi_{k+1})\mid \Phi_k = z'] \bigr|}{\sqrt{V(z')}}\sup_{z''\in\zstate}\frac{\bigl| \Expect[\hahaf_j(\Phi_{n+2})\mid \Phi_{k+1}=z''] \bigr|}{\sqrt{V(z'')}} \leq B_{\sqrtV}^2\|\haf_i\|_{\sqrtV}\|\hahaf_j\|_{\sqrtV} \rho^{n+2-k}
\end{aligned}
\label{e:Vstar-2}
\end{equation}
Combining \eqref{e:Vstar-1} and \eqref{e:Vstar-2} gives
\begin{equation}
\|h_{i,j}^{k+1,n+2}\|_{V^*} \leq B\|\haf_i\|_{\sqrtV}\|\hahaf_j\|_{\sqrtV} \rho^{n+1-k}
\label{e:Vuni-W}
\end{equation}
with $B=B_{\sqrtV}+B_{\sqrtV}^2$. 

For (ii), denote $g_{i,j}^{k,n+2}:\zstate\rightarrow\Re$ by the conditional expectation:
\[
g_{i,j}^{k,n+2}(z)= \Expect[h_{i,j}^{k+1,n+2}(\Phi_k, \Phi_{k+1})\mid \Phi_k=z]
\]
This is bounded by a constant times $V^*$:
\[ 
\begin{aligned}
|g_{i,j}^{k,n+2}(z)| = \Bigl| \int h_{i,j}^{k+1,n+2}(z,z')P(z,dz') \Bigr| & \leq \Bigl| \int \frac{h_{i,j}^{k+1,n+2}(z,z')}{V^*(z,z')}V^*(z,z')P(z,dz') \Bigr| \\
																										& \leq \|h_{i,j}^{k+1,n+2}\|_{V^*} [V(z) + PV(z)]
\end{aligned}
\]
 $V$-uniform ergodicity of $\bfPhi$   is equivalent to the following drift condition~\cite[Theorem 16.0.2]{MT}: 
 for some $\beta>0, b < \infty$, and some ``petite set'' $C$,
\[
PV(z) - V(z) \leq -\beta V(z) + b\ind_C(z)\,, \qquad z\in \zstate
\]
Consequently, 
\[
 [V(z) + PV(z)]\leq [2V(z) +b] \leq [2+|b|]V(z)
\]
Therefore,
\begin{equation}
\|g_{i,j}^{k,n+2}\|_V \leq [2+|b|]\|h_{i,j}^{k+1,n+2}\|_{V^*} \leq [2+|b|]B\|\haf_i\|_{\sqrtV}\|\hahaf_j\|_{\sqrtV} \rho^{n+1-k}
\label{e:V-norm-g}
\end{equation}
Thus $g_{i,j}^{k,n+2}\in \LV$. By $V$-uniform ergodicity of $\bfPhi$ again, 
\[
\begin{aligned}
\bigl|\Expect[g_{i,j}^{k,n+2}(\Phi_k)\mid \Phi_0 =z] -\pi\bigl(g_{i,j}^{k,n+2}\bigr) \bigr| 	&\leq B_V\|g_{i,j}^{k,n+2}\|_{V}V(z)\rho^k \\
																										&\leq B'\|\haf_i\|_{\sqrtV}\|\hahaf_j\|_{\sqrtV} V(z) \rho^{n+1}
\end{aligned}
\]
with $B'= [2+|b|]B_VB$. The proof is then completed by applying the smoothing property of conditional expectation.
\end{proof}

\begin{lemma}
\label{t:V-geo-secondPoi}
Under Assumptions (A1) and (A3), there exists   $K<\infty$ such that the following hold
\begin{subequations}%
\begin{align}
\bigl\| \Expect[\Delta_{k+1}^m \haZ_{n+3}^\transpose] \bigr\|_T 
	&\leq K \rho^{n+1-k}   \label{e:V-geo-secondPoi-1}\\
\bigl\| \Expect[\Delta_{k+1}^m \haZ_{n+2}^\transpose] - \Expect[\Delta_{k+2}^m \haZ_{n+3}^\transpose] \bigr\|_T
	&\leq K(1+\rho)\rho^{n+1}
\label{e:V-geo-secondPoi-2}
\end{align}%
\label{e:V-geo-secondPoi}%
\end{subequations}%
\end{lemma}
\begin{proof}
By the triangle inequality,
\[
 \bigl\|\Expect[\Delta_{k+1}^m \haZ_{n+2}^\transpose] \bigr\|_T \leq \bigl\|\Expect[Z_{k+1}\haZ_{n+2}^\transpose] \bigr\|_T + \bigl\|\Expect\bigl[\Expect[Z_{k+1}|\clF_{k}] \haZ_{n+2}^\transpose\bigr] \bigr\|_T
\]
where both terms admit the geometric bound in \eqref{e:V-geo-secondPoi-1} following directly from the $V$-geometric mixing of $\bfPhi$ ~\cite[Theorem 16.1.5]{MT}. 

For \eqref{e:V-geo-secondPoi-2}, first notice that
\[
\Expect[\Delta_{k+1}^m \haZ_{n+2}^\transpose] =\Expect\bigl[\Expect[\Delta_{k+1}^m \haZ_{n+2}^\transpose \mid \clF_{k+1}]\bigr] = \Expect[h^{k+1,n+2}(\Phi_{k}, \Phi_{k+1})]
\]
With \Lemma{t:bivariate-chain}, we have for each $(i,j)$-th entry,
\[
\Bigl|\Expect[h_{i,j}^{k+1,n+2}(\Phi_k, \Phi_{k+1})\mid \Phi_0 =z] - \pi\bigl(h_{i,j}^{k+1,n+2}\bigr)\Bigr| 
 \leq B'\|\haf_i\|_{\sqrtV}\|\hahaf_j\|_{\sqrtV} V(z) \rho^{n+1}
\]
With fixed initial condition $\Phi_0=z$, by equivalence of matrix norms, there exists a constant $K$ such that
\[
\Bigl\| \Expect[h^{k+1,n+2}(\Phi_{k}, \Phi_{k+1})] -\pi\bigl(h_{i,j}^{k+1,n+2}\bigr) \Bigr\|_T \leq K\rho^{n+1}
\]
\eqref{e:V-geo-secondPoi-2} then follows from the triangle inequality:
\[
\bigl\| \Expect[\Delta_{k+1}^m \haZ_{n+2}^\transpose] - \Expect[\Delta_{k+2}^m \haZ_{n+3}^\transpose] \bigr\|_T \leq K\rho^{n+1} + K\rho^{n+2} = K(1+\rho)\rho^{n+1}
\]
\end{proof}

\begin{lemma}
\label{e:exp-sum}
For fixed $\rho \in (0,1)$, there exists   $K<\infty $ such that for all $n \ge 2$,
\[
\sum_{k=1}^{n-1} \frac{1}{k}\rho^{-k} \leq K\frac{\rho^{-n}}{n}
\]
\end{lemma}
\begin{proof}
Denote $\gamma = -\log\rho > 0$ and observe that the  function $t^{-1}\exp(\gamma t)$ is increasing over $[1,\infty)$. The following holds for $n\geq 2$
\[
\sum_{k=1}^{n-1} \frac{1}{k}\rho^{-k} = \sum_{k=1}^{n-1} \frac{1}{k}\exp(\gamma k) \leq \int_1^{n} t^{-1}\exp(\gamma t)dt
\]
Now consider the integral: for any $t_0\in(1, n)$,
\[
\begin{aligned}
\int_1^{n} t^{-1}\exp(\gamma t)dt 
												&  \leq \int_1^{t_0} \exp(\gamma t)dt + \int_{t_0}^{n} t_0^{-1}\exp(\gamma t)dt \\
												& \leq \gamma^{-1}\bigl[\exp(\gamma t_0) - \exp(\gamma) + \frac{\exp(\gamma n) - \exp(\gamma t_0)}{t_0}\bigr]
\end{aligned}
\] 
Take $t_0= n - \sqrt{n}$.
\[
\begin{aligned}
\exp(\gamma t_0) - \exp(\gamma) + \frac{\exp(\gamma n) - \exp(\gamma t_0)}{t_0} 
&= \exp(\gamma (n - \sqrt{n})) - \exp(\gamma) + \frac{\exp(\gamma n)- \exp(\gamma(n-\sqrt{n}))}{n-\sqrt{n}} \\
&\leq K' n^{-1}\exp(\gamma n)
\end{aligned}
\]
where $K' = \sup_{t \ge 2} t\exp(-\gamma \sqrt{t}) - t\exp(\gamma-\gamma t) + [1-\exp(-\gamma\sqrt{t})]/[1-1/\sqrt{t}]$. The proof is completed by setting $K=\gamma^{-1} K'$.
\end{proof}

\begin{proof}[Proof of \Proposition{t:fine-comp} (iii)]
Following \eqref{e:cov-theta-mart}, we have
\begin{equation}
R_{n+1}^{(1),(3)}=\Expect[\tiltheta_{n+1}^{(1)}(\tiltheta_{n+1}^{(3)})^\transpose]= \frac{1}{n+1}\Expect[\tiltheta_{n+1}^{(1)}[\haZ_{n+3}- \haZ_{n+2}]^\transpose]
\label{e:corr-13-decomp}
\end{equation}
This is bounded based on \eqref{e:cov-mart-Z}: \Lemma{t:prod-matrix} and \eqref{e:V-geo-secondPoi-2} indicate that there exists some constant $K$ such that
\begin{equation}
\begin{aligned}
 \sum_{k=1}^{n+1} \alpha_k \bigl\|\prod_{l=k+1}^{n+1}[I+\alpha_lA]\bigr\|_T  \bigl\|\Expect\bigl[\Delta_{k+2}^m\haZ_{n+3}^\transpose -  \Delta_{k+1}^m\haZ_{n+2}^\transpose \bigr] \bigr\|_T \leq K \rho^{n+1} 
\end{aligned}
\label{e:corr-cle-z-1}
\end{equation}
For the second term in \eqref{e:cov-mart-Z}, it admits a simpler form
\[
\begin{aligned}
\sum_{k=1}^{n+1} \alpha_k\prod_{l=k+1}^{n+1}[I+\alpha_lA] \bigl[\Delta_{k+1}^m - \Delta_{k+2}^m \bigr]\haZ_{n+3}^\transpose
= & \prod_{l=2}^{n+1}[I+\alpha_lA]\Delta_2^m \haZ_{n+3}^\transpose - \frac{1}{n+1} \Delta_{n+3}^m \haZ_{n+3}^\transpose \\
&- \sum_{k=2}^{n+1}\alpha_{k-1} \alpha_k \prod_{l=k+1}^{n+1}[I + \alpha_l A] [I+A]\Delta_{k+1}^m \haZ_{n+3}^\transpose
\end{aligned}
\]
where $\prod_{l=2}^{n+1}[I+\alpha_lA]\Expect[\Delta_2 \haZ_{n+3}^\transpose] = O(\rho^n)$ and $\Expect[\Delta_{n+3}^m \haZ_{n+3}^\transpose]$ converges to its steady-state mean. For the remaining part, \Lemma{t:prod-matrix} and \eqref{e:V-geo-secondPoi-1} together imply that 
\[
\begin{aligned}
&\Bigl\| \sum_{k=2}^{n+1} \alpha_{k-1} \alpha_k \prod_{l=k+1}^{n+1}[I + \alpha_l A] [I+A]\Expect[\Delta_{k+1}^m \haZ_{n+3}^\transpose] \Bigr\|_T \\
& \leq \sum_{k=2}^{n+1} \alpha_{k-1} \alpha_k \prod_{l=k+1}^{n+1}\| I + \alpha_l A\|_T \|I+A\|_T \|\Expect[\Delta_{k+1}^m \haZ_{n+3}^\transpose] \|_T \\
&\leq \frac{K'}{n+2}\sum_{k=2}^{n+1} \alpha_{k-1} \rho^{n+1-k}
\end{aligned}
\]
for some constant $K'$. By \Lemma{e:exp-sum}, there exists another constant $K''$ such that
\[
 \frac{K'}{n+2}\sum_{k=2}^{n+1} \alpha_{k-1} \rho^{n-k} = \frac{K'\rho^{n}}{n+2}\sum_{k=1}^{n} \alpha_{k} \rho^{-k}  \leq  \frac{K'K''\rho}{(n+1)(n+2)}
\]
This combined with \eqref{e:corr-cle-z-1} shows that 
\[
\Expect[\tiltheta_{n+1}^{(1)}[\haZ_{n+3} -\haZ_{n+2}]^\transpose] = -(n+1)^{-1}\Expect_{\pi}[\Delta_{n}^m\haZ_{n}^\transpose] +O(\rho^{n+1})
\]
Following \eqref{e:corr-13-decomp}, we obtain the desired result:
\[
\Expect[\tiltheta_{n+1}^{(1)}(\tiltheta_{n+1}^{(3)})^\transpose] =  -\frac{1}{(n+1)^2}\Expect_{\pi}[\Delta_{n}^m\haZ_{n}^\transpose]  +O((n+1)^{-3})
\]
\end{proof}

\subsection{Unbounded moments}
\label{s:unbounded}

This section is devoted to the proof that $\lim_{n\rightarrow\infty} \Expect[|v^\transpose\tiltheta_n^\varrho|^2] = \infty$ for $\varrho> \varrho_0$ (see  \Theorem{t:SAlinearized} (ii)). Since it suffices to show the result holds for $\varrho_0 < \varrho < \half$, we assume $\varrho<\half$ throughout. Recall that $\lambda = -\varrho_0 + ui$.

Consider the update of $\tiltheta_n^\varrho$ in \eqref{e:clElin}. With $v^\transpose [\lambda I - A] = 0$, we have $v^\transpose [\varrho_n I + A_n] = v^\transpose[\varrho - \varrho_0 +\epsy_v(n,\varrho) + ui]$. Multiplying each side of \eqref{e:clElin} by $v^\transpose$ gives
\[
\begin{aligned}
v^\transpose\tiltheta_{n+1}^\varrho &= v^\transpose\tiltheta_n^\varrho  + \alpha_{n+1}\bigl[ [\varrho - \varrho_0 + \epsy_v(n, \varrho) + ui]v^\transpose\tiltheta_n^\varrho + (n+1)^\varrho v^\transpose\Delta_{n+1} \bigr] \\
&=[ 1+\alpha_{n+1}\tilde\varrho_{n+1} + \alpha_{n+1}ui]v^\transpose\tiltheta_n^\varrho + (n+1)^{\varrho-1} v^\transpose\Delta_{n+1}
\end{aligned}
\]
with $\tilde\varrho_{n+1} = \varrho - \varrho_0 + \epsy_v(n, \varrho)$. Note that $\tilde\varrho_{n+1}$ is strictly positive for sufficiently large $n$.

For a fixed but arbitrary $n_0$ and each $n \ge n_0$, we have
\begin{equation}
\label{e:clE-iterate}
\begin{aligned}
v^\transpose\tiltheta_{n+1}^\varrho &= v^\transpose\tiltheta_{n_0}^{\varrho} \prod_{k=n_0+1}^{n+1} [1+\alpha_{k}\tilde\varrho_{k} + \alpha_kui]  + \sum_{k=n_0+1}^{n+1} k^{\varrho-1}v^\transpose \Delta_{k} \prod_{l=k+1}^{n+1}[1+\alpha_l\tilde\varrho_l + \alpha_lui] \\
&= \Bigl[\prod_{k=n_0+1}^{n+1} [1+\alpha_{k}\tilde\varrho_{k} + \alpha_kui]\Bigr] \cdot \Bigl[v^\transpose\tiltheta_{n_0}^\varrho + \sum_{k=n_0+1}^{n+1} \frac{k^{\varrho-1} }{\prod_{l =n_0+1}^k [1+\alpha_l\tilde\varrho_l+ \alpha_lui]} v^\transpose \Delta_k  \Bigr] \\
&= \Bigl[\prod_{k=n_0+1}^{n+1} [1+\alpha_{k}\tilde\varrho_{k} + \alpha_kui]\Bigr] \cdot \Bigl[v^\transpose\tiltheta_{n_0}^\varrho + \sum_{k=n_0+1}^{n+1} \beta_k v^\transpose \Delta_k  \Bigr]
\end{aligned}
\end{equation}
with $\beta_n = n^{\varrho-1}/\prod_{l=n_0+1}^n [1+\alpha_l\tilde\varrho_l + \alpha_l ui]$.

The analysis of $\{v^\transpose\tiltheta_n^{\varrho}\}$ is mainly based on the random series appearing in \eqref{e:clE-iterate}, which requires the following three preliminary results: 
\begin{lemma}
	\label{t:beta}
	There exists some $n_0$ such that for each $n> n_0$,
\[
	|\beta_n - \beta_{n+1}|^2 \leq  4|\beta_{n+1}|^2\alpha_n^2 (1 + u^2)
\]
\end{lemma}

\begin{proof}
	Note that $|\beta_n - \beta_{n+1}|^2 = |\beta_{n+1}|^2| \beta_n /\beta_{n+1} - 1|^2$, so it is sufficient to bound the second factor:
\begin{equation}
	\label{e:beta-diff}
	\begin{aligned}
	| \beta_n /\beta_{n+1} - 1|^2 &=  |(1 + n^{-1})^{1-\varrho}[1+\alpha_{n+1}\tilde\varrho_{n+1} + \alpha_{n+1}ui] - 1|^2 \\
	&= |(1 + n^{-1})^{1-\varrho}[1+\alpha_{n+1}\tilde\varrho_{n+1}] -1 + (1 + n^{-1})^{1-\varrho}\alpha_{n+1}ui|^2
	\end{aligned}
\end{equation}
	Consider the real part in \eqref{e:beta-diff}: since $\epsy_v(n, \varrho) = O(n^{-1})$, there exists $n_0$ such that $|\epsy_v(n, \varrho)| \leq \varrho- \varrho_0$ and $\tilde\varrho_{n+1}=\varrho - \varrho_0 + \epsy_v(n,\varrho) > 0$ for $n \geq n_0$.
	Consequently,
\[
	\begin{aligned}
	0 \leq (1 + n^{-1})^{1-\varrho}[1 + \alpha_{n+1}\tilde\varrho_{n+1}] - 1 &< (1 + n^{-1})[1 + \alpha_{n+1}\tilde\varrho_{n+1}]- 1 \\
	& \leq n^{-1}(1 + \tilde\varrho_{n+1} + \alpha_{n+1}\tilde\varrho_{n+1})
	\end{aligned}
\]
	Given $0<\varrho -\varrho_0< \half$,  we can increase $n_0$ if necessary, such that $1 + \tilde\varrho_{n+1} + \alpha_{n+1}\tilde\varrho_{n+1} \leq 2$ for $n\geq n_0$. Then we have
\[
	(1+n^{-1})^{1-\varrho}[1+\alpha_{n+1}\tilde\varrho_{n+1}] - 1 \leq 2\alpha_n
\]
	For the imaginary part, observe that
\[
	(1+n^{-1})^{1-\varrho}\alpha_{n+1}u = \alpha_n \frac{n^\varrho}{(n+1)^\varrho}u\leq 2u\alpha_n
\]
	The proof is completed by summing the bounds for the real and imaginary parts.
\end{proof}

\begin{lemma}
	\label{t:rs-convergence}
	Suppose Assumptions A1 and A3 hold. With each $n_0\geq 1$, the random series $\sum_{k=n_0+1}^{\infty} \beta_k v^\transpose \Delta_k$ converges a.s..
\end{lemma}

\begin{proof}
	Decompose the series into the sum of a martingale difference and telescoping sequence. The martingale difference sequence converges \textit{almost surely} given $\{\beta_n\}\in\ell_2$; the telescoping series is absolutely convergent by \Lemma{t:beta}.
\end{proof}

\begin{lemma}
	\label{t:var-bound}
	Suppose Assumptions A1 and A3 hold. Denote $Z_n^v =v^\transpose Z_n= v^\transpose\haf(\Phi_n)$. There exists a deterministic constant $K>0$, such that for all $n_0$ and each sequence $\gamma\in\ell_1 \subseteq \ell_2$,
\begin{equation}
	\Expect\bigl[\Var(\sum_{k=n_0+2}^\infty \gamma_{k-n_0-1}Z_k^v\mid \clF_{n_0+1})\bigr] \leq K \sum_{k=1}^\infty |\gamma_k|^2
\end{equation}
\end{lemma}

\begin{proof}
	First recall that
$	\displaystyle
	\Var\bigl(\sum_{k=n_0+2}^\infty \gamma_{k-n_0-1}Z_k^v \mid \clF_{n_0+1}\bigr) \leq \Expect\bigl[|\sum_{k=n_0+2}^\infty \gamma_{k-n_0-1}Z_k^v|^2 \mid \clF_{n_0+1}\bigr] 
	$, and hence by the Markov property, 
\[
	\Expect\bigl[|\sum_{k=n_0+2}^\infty \gamma_{k-n_0-1}Z_k^v|^2 \mid \clF_{n_0+1}\bigr] 
	= \Expect_{z'}\bigl[|\sum_{k=1}^\infty \gamma_{k}Z_k^v|^2 \bigr] = \lim_{n\rightarrow\infty} \Expect_{z'}\bigl[|\sum_{k=1}^n \gamma_{k}Z_k^v|^2 \bigr] 
\]
	where $z' = \Phi_n$, and the last equality holds by the assumption $\gamma \in\ell_1$ and dominated convergence.  
	\sh{should we change this part $\gamma \in \ell_1$ ?}	
	For each $n$, letting $\lceil\gamma\rceil^n = (\gamma_1,\dots,\gamma_n)$ denote $\gamma$ truncated at index $n$, we have
\begin{equation}
	\label{e:cov-quad-form}
	\Expect_{z'}\bigl[|\sum_{k=1}^n \gamma_{k}Z_k^v|^2 \bigr] = \sum_{k=1}^n|\gamma_k|^2\Expect_{z'}\bigl[|Z_k^v|^2\bigr] + \sum_{i=1}^n\sum_{j\neq i}^n \gamma_i^\dagger \gamma_j \Expect_{z'}\bigl[(Z_i^v)^\dagger Z_j^v\bigr] = (\lceil\gamma\rceil^n)^\dag [R]_n \lceil\gamma\rceil^n
\end{equation}
	where $[R]_n\in \mathbb{C}^{n\times n}$ is the covariance matrix with each entry defined as $R(i,j) = \Expect_{z'}\bigl[(Z_i^v)^\dagger Z_j^v\bigr], 1\leq i,j \leq n$; $[R]_n$ is Hermitian and positive semi-definite. With $\lambda_n \geq 0$ denoting the largest eigenvalue of $[R]_n$, we have
\begin{equation}
	\label{e:cov-quad-eigen}
	(\lceil\gamma\rceil^n)^\dag [R]_n \lceil\gamma\rceil^n \leq \lambda_n \sum_{k=1}^n |\gamma_k|^2 \leq \lambda_n\sum_{k=1}^\infty |\gamma_k|^2
\end{equation}
	By the Gershgorin circle theorem~\cite{golvan1996matrix}, the maximal eigenvalue is upper bounded by the maximum row sum of absolute values of entries:
\[
	\lambda_n \leq \max_{i\in\{1,\dots,n\}} \sum_{j=1}^n | R(i,j)| \leq \sup_{i\in\mathbb{Z}_+} \sum_{j=1}^\infty |R(i, j)|
\]
	For any $i$, observe that
\[
	\sum_{j=1}^{\infty} |R(i,j)| = \Expect_{z'}\bigl[|Z_i^v|^2\bigr] + \sum_{i<j} |R(i,j)| + \sum_{i > j} |R(i,j)|
\]
	Since $V$-uniform ergodicity of the Markov chain $\bfPhi$ implies $V$-geometric mixing~\cite[Theorem 16.1.5]{MT} and $|v^\transpose\haf|^2 \in \LV$, there exist $B<\infty$ and $r\in(0,1)$ such that for each $i, k \in \mathbb Z_+$,
\[
	\Bigl| R(i, i+k) - \Expect_{z'}\bigl[(Z_i^v)^\dagger\bigr]\Expect_{z'}\bigl[Z_{i+k}^v\bigr]\Bigl| \leq B r^k[1+r^iV(z')]
\]
	Consequently,
\begin{equation}
	\label{e:cov-absSum}
	\begin{aligned}
	\sum_{j=1}^{\infty} |R(i,j)| \leq \, & \Expect_{z'}\bigl[|Z_i^v|^2\bigr] + \Bigl|\Expect_{z'}\bigl[(Z_i^v)^\dagger\bigr]\Bigr| \sum_{j=1}^\infty \Bigl|\Expect_{z'}[Z_j^v]\Bigr| \\
	&+ \sum_{i<j} Br^{j-i}[1+r^i V(z')] + \sum_{i>j} Br^{i-j}[1+r^j V(z')] 
	\end{aligned}
\end{equation}
	Given $|v^\transpose\haf|^2\in\LV$, by \eqref{e:Vuni},
\[
	\Expect_{z'}\bigl[|Z_n^v|^2\bigr] \leq \Expect_{\pi}\bigl[|Z_n^v|^2\bigr] + B_V \bigl\| |v^\transpose\haf|^2 \bigr\|_V V(z')
\]
	The Markov chain $\bfPhi$ is also  $\sqrt{V}$-uniformly ergodic. By \eqref{e:Vuni} for $\sqrt{V}$ and $|v^\transpose\haf|^2\in\LV$ once more, 
\[
	\bigl|\Expect_{z'}[(Z_i^v)^\dagger]\bigr| \leq B_{\scriptscriptstyle\sqrt{V}}\|v^\transpose\haf\|_{\scriptscriptstyle\sqrt{V}} \sqrt{V(z')}\rho^j
\]
	Hence
\[
	\bigl|\Expect_{z'}[(Z_i^v)^\dagger]\bigr| \sum_{j=1}^\infty \bigl|\Expect_{z'}[Z_j^v]\bigr| \leq B_{\scriptscriptstyle\sqrt{V}}^2\|v^\transpose\haf\|^2_{\scriptscriptstyle\sqrt{V}} V(z')\rho^i \sum_{j=1}^\infty \rho^j \leq B^2_{\scriptscriptstyle\sqrt{V}}\|v^\transpose\haf\|^2_{\scriptscriptstyle\sqrt{V}}  \frac{\rho}{1-\rho}V(z')
\]
	The other two terms on the right hand side of \eqref{e:cov-absSum} are bounded as follows:
\[
	\begin{aligned}
	\sum_{j>i} B r^{j-i}[1+r^i V(z')]  &= \sum_{j>i} B [r^{j-i}+r^j V(z')] \leq \frac{Br}{1-r}(1+V(z')) \\
	\sum_{j<i} B r^{i-j}[1+r^j V(z')] &= \bigl[\sum_{j<i} B [r^{i-j}]\bigr] + BV(z') (i-1) r^i \leq \frac{Br}{1-r} + BV(z')\sup_i i r^i
	\end{aligned}
\]
	where $\sup_i ir^i$ exists since $\lim_{n\rightarrow\infty} nr^n = 0$. 
	
	Consequently, there exists some deterministic constant $K'$ independent of $z'$ such that, the largest eigenvalues $\{\lambda_n\}$ are uniformly bounded
\[
	\sup_n \lambda_n \leq K'V(z')
\] 
	Combining this with \eqref{e:cov-quad-form} and \eqref{e:cov-quad-eigen} gives
\[
	\Expect_{z'}\bigl[|\sum_{k=1}^\infty Z_k^v|^2\bigr] \leq K'V(z') \sum_{k=1}^\infty |\gamma_k|^2
\]
	Therefore,
\[
	\Expect\Bigl[\Expect\bigl[|\sum_{k=n_0+2}^\infty \gamma_{k-n_0-1}Z_k^v|^2 \mid \clF_{n_0+1}\bigr] \mid \Phi_0 = z\Bigr] \leq K'\Expect\bigl[V(\Phi_{n_0+1})\mid \Phi_0 =z \bigr] \sum_{k=1}^\infty |\gamma_k|^2 
\]
	By $V\in \LV$ and \eqref{e:Vuni} again, $\Expect[V(\Phi_{n_0+1})\mid \Phi_0 =z] \leq \pi(V) + B_V V(z)$. The desired conclusion then follows by setting $K = K'(B_V V(z) + \pi(V))$.
\end{proof}

\begin{lemma}
	\label{t:unstable}
	Suppose Assumptions A1-A3 hold and $\Sigma_{\Delta}v \neq 0$. With $\{\tiltheta_n^\varrho\}$ updated via \eqref{e:clElin},
\[
	\liminf_{n\rightarrow \infty} \Expect[|v^\transpose \tiltheta_n^{\varrho}|^2] = \infty\,, \qquad \varrho > \varrho_0
\]
\end{lemma}

\begin{proof}
	
	With fixed $n_0$, equation \eqref{e:clE-iterate} gives a representation for $v^\transpose\tiltheta_{n+1}^\varrho$ for each $n\ge n_0$. It is obvious that $\liminf_{n\rightarrow\infty} \prod_{k=n_0+1}^n |1 + \tilde{\varrho}_k\alpha_k + \alpha_k ui|^2=\infty$. Hence it suffices to show that $\liminf_{n\rightarrow\infty} \Expect[|v^\transpose \tiltheta_{n_0}^\varrho + \sum_{k=n_0+1}^{n+1} \beta_k v^\transpose \Delta_k|^2]$ is strictly greater than zero. 
	
	By Fatou's lemma,
\[
	\begin{aligned}
	\liminf_{n\rightarrow\infty} 
	\Expect\bigl[|v^\transpose \tiltheta_{n_0}^\varrho + \sum_{k=n_0+1}^{n+1} \beta_k v^\transpose \Delta_k|^2\bigr] 
	&\geq \Expect\bigl[\liminf_{n\rightarrow\infty} |v^\transpose \tiltheta_{n_0}^\varrho + \sum_{k=n_0+1}^{n+1} \beta_k v^\transpose \Delta_k|^2\bigr] \\
	& = \Expect\bigl[|v^\transpose \tiltheta_{n_0}^\varrho + \sum_{k=n_0+1}^{\infty} \beta_k v^\transpose \Delta_k|^2\bigr]\\
	&\geq \Var\bigl(v^\transpose \tiltheta_{n_0}^\varrho + \sum_{k=n_0+1}^{\infty} \beta_k v^\transpose \Delta_k\bigr)
	\end{aligned}
\]
	where the equality holds by \Lemma{t:rs-convergence}. 
	By the law of total variance,
\[
	\begin{aligned}
	\Var\bigl(v^\transpose \tiltheta_{n_0}^\varrho + \sum_{k=n_0+1}^{\infty} \beta_k v^\transpose \Delta_k\bigr) & \geq \Expect\bigl[\Var(v^\transpose \tiltheta_{n_0}^\varrho + \sum_{k=n_0+1}^{\infty} \beta_k v^\transpose \Delta_k \mid \clF_{n_0+1 })\bigr]  \\
	& = \Expect\bigl[\Var(\sum_{k=n_0+1}^{\infty} \beta_k v^\transpose \Delta_k \mid \clF_{n_0+1})\bigr]
	\end{aligned}
\]
	
	Apply once more the decomposition based on Poisson's equation:
\[
	v^\transpose \Delta_n = \Delta_{n+1}^{vm} + Z_n^v - Z_{n+1}^v\,, \qquad n\geq 1 \,,
\]
	where $Z_n^v = v^\transpose\haf(\Phi_n)$ and $\Delta_{n+1}^{vm} = Z_{n+1}^v - \Expect[Z_{n+1}^v\mid\clF_n]$ is a martingale difference. By the variance inequality $\Var(X+Y\mid\clF_{n_0+1}) \leq 2\Var(X\mid\clF_{n_0+1}) + 2\Var(Y\mid\clF_{n_0+1})$, we have
\begin{equation}
	\label{e:var-decomp}
	\begin{aligned}
	\Expect\bigl[\Var &(\sum_{k=n_0+1}^{\infty} \beta_k v^\transpose \Delta_k \mid  \clF_{n_0+1})\bigr] \\
	& \geq 
	\half \Expect\bigl[\Var(\sum_{k=n_0+1}^{\infty} \beta_k \Delta_{k+1}^{vm} \mid \clF_{n_0+1})\bigr]
	- \Expect\bigl[\Var(\sum_{k=n_0+1}^{\infty} \beta_k (Z_k^v - Z_{k+1}^v) \mid\clF_{n_0+1})\bigr]
	\end{aligned}
\end{equation}

By the law of total variance once more, 
\[
	\Var\bigl(\sum_{k=n_0+1}^\infty \beta_k \Delta_{k+1}^{vm}\bigr) = \Expect\bigl[\Var(\sum_{k=n_0+1}^\infty \beta_k \Delta_{k+1}^{vm}\mid \clF_{n_0+1})\bigr] 
	+ \Var\bigl(\Expect[\sum_{k=n_0+1}^\infty \beta_k\Delta_{k+1}^{vm} \mid \clF_{n_0+1}]\bigr) 
\]
Note that $\lim_{n\to\infty}\Expect[\sum_{k=n_0+1}^n \beta_k\Delta_{k+1}^{vm} \mid \clF_{n_0+1}]$ converges to zero \emph{almost surely}. With $\{\beta_n\} \in\ell_2$ and the Jensen's inequality, we have for all $n$, 
	\[
	 \bigl|\Expect[\sum_{k=n_0+1}^n \beta_k\Delta_{k+1}^{vm} \mid \clF_{n_0+1}]\bigr|^2 \leq \sum_{k=n_0+1}^{\infty} |\beta_k|^2 \Expect[|\Delta_{k+1}^{vm}|^2 \mid \clF_{n_0+1}] <  \infty
	\]
Then by the dominated convergence theorem, $\Expect\bigl[\bigl|\Expect[\sum_{k=n_0+1}^\infty \beta_k\Delta_{k+1}^{vm} \mid \clF_{n_0+1}]\bigr|^2\bigr] = 0$. Therefore, 
\[
	\Var\bigl(\Expect[\sum_{k=n_0+1}^\infty \beta_k\Delta_{k+1}^{vm} \mid \clF_{n_0+1}]\bigr)  \leq \Expect\bigl[\bigl|\Expect[\sum_{k=n_0+1}^\infty \beta_k\Delta_{k+1}^{vm} \mid \clF_{n_0+1}]\bigr|^2\bigr] = 0
\]

	Hence,
\begin{equation}
	\label{e:varMD}
	\Expect\bigl[\Var(\sum_{k=n_0+1}^\infty \beta_k \Delta_{k+1}^{vm}\mid\clF_{n_0+1})\bigr] =\Var\bigl(\sum_{k=n_0+1}^\infty \beta_k \Delta_{k+1}^{vm}\bigr)  = \sum_{k=n_0+1}^\infty |\beta_k|^2\sigma_{k+1}^2
\end{equation}
	where $\sigma_n^2 = \Var(\Delta_n^{vm})$. 
	
	For the telescoping term on the right hand side of \eqref{e:var-decomp}, we have
\begin{equation}
	\label{e:varTele}
	\begin{aligned}
	\Expect\bigl[\Var(\sum_{k=n_0+1}^{\infty} \beta_k (Z_k^v - Z_{k+1}^v) \mid\clF_{n_0+1})\bigr] &= \Expect\bigl[\Var( \beta_{n_0+1}Z_{n_0+1}^v - \sum_{k=n_0+2}^{\infty} (\beta_{k} - \beta_{k+1})Z_k^v \mid \clF_{n_0+1})\bigr] \\
	&= \Expect\bigl[\Var( \sum_{k=n_0+2}^{\infty} (\beta_{k} - \beta_{k+1})Z_k^v \mid \clF_{n_0+1})\bigr]
	\end{aligned}
\end{equation}
	Given $\{\beta_n - \beta_{n+1}\}\in\ell_1$ by \Lemma{t:beta}, \Lemma{t:var-bound} indicates that there exists some constant $K$ independent of $n_0$ such that,
\[
	\Expect\bigl[\Var( \sum_{k=n_0+2}^{\infty} (\beta_{k} - \beta_{k+1})\hat{Z}_k \mid \clF_{n_0+1})\bigr] \leq K \sum_{k=n_0+2}^\infty |\beta_k - \beta_{k+1}|^2
\]
	Combining \eqref{e:varMD} and \eqref{e:varTele} gives
\[
	\begin{aligned}
	\Expect[\Var(\sum_{k=n_0+1}^{\infty} \beta_k v^\transpose \Delta_k \mid \clF_{n_0+1})] \geq \half \sum_{k=n_0+1}^\infty |\beta_k|^2\sigma_{k+1}^2 - K \sum_{k=n_0+2}^\infty |\beta_k - \beta_{k+1}|^2
	\end{aligned}
\]
	Since $|v^\transpose\haf|^2\in\LV$ and $\sigma_n^2 \rightarrow \sigma^2=v^\transpose\Sigma_\Delta \overline{v}>0$ at a geometric rate, we set $n_0$ sufficiently large such that \Lemma{t:beta} holds and moreover for all $n\geq n_0$,
\[
	\sigma_n^2 \geq \half \sigma^2,\qquad \frac{1}{4}\sigma^2 - 4K\alpha_n^2(1+u^2) \geq \frac{1}{8}\sigma^2 \,, 
\]
	Then,
\[
	\Expect\bigl[\Var(\sum_{k=n_0+1}^{\infty} \beta_k v^\transpose \Delta_k \mid \clF_{n_0+1})\bigr] \geq \frac{1}{8}\sigma^2 \sum_{k=n_0+1}^\infty |\beta_k|^2 
\]
	Therefore,
\[
	\liminf_{n\rightarrow\infty} \Expect\bigl[|v^\transpose \tiltheta_{n_0}^\varrho + \sum_{k=n_0+1}^n \beta_k v^\transpose\Delta_k|^2\bigr] \geq \frac{1}{8}\sigma^2 \sum_{k=n_0+1}^\infty |\beta_k|^2 >0
\]
	The desired conclusion then follows from \eqref{e:clE-iterate}:
\[
	\liminf_{n\rightarrow\infty} \Expect\bigl[|v^\transpose\tiltheta_{n+1}^\varrho|^2\bigr] \geq \liminf_{n\rightarrow\infty} \prod_{k=n_0+1}^n |1 + \tilde{\varrho}_k\alpha_k + \alpha_k ui|^2 \cdot\liminf_{n\rightarrow\infty}  \Expect\bigl[|v^\transpose \tiltheta_{n_0}^\varrho + \sum_{k=n_0+1}^n \beta_k v^\transpose\Delta_k|^2\bigr] = \infty
\]

\end{proof}

\subsection{Coupling of Deterministic and Random Linear SA}
\label{s:couple-td}

Let $\haclA:\zstate\rightarrow \Re^{d\times d}$ denote the zero-mean solution to the following Poisson equation:
\[
\Expect[\haclA(\Phi_{n+1})\mid \Phi_n = z] = \haclA(z) - \clA(z) + A \,, \qquad z\in\zstate 
\]
which is a matrix version of \eqref{e:fishDef}. Denote  $\Delta^{\clA}_{n+1} = \haclA(\Phi_{n+1}) - \Expect[\haclA(\Phi_{n+1})\mid \clF_n]$  (a martingale difference sequence), and $\clA_n = \haclA(\Phi_n)$. Then, from \eqref{e:linSAerr_An},
\[
\begin{aligned}
(A_{n+1}-A)\tiltheta^\circ_n & =  [  \Delta^{\clA}_{n+2} + \clA_{n+1} - \clA_{n+2} ] \tiltheta^\circ_n 
\\	
								 & = \Delta^{\clA}_{n+2}\tiltheta^\circ_n + \clA_{n+1}\tiltheta^\circ_n - \clA_{n+2}\tiltheta^\circ_{n+1} + \clA_{n+2}(\tiltheta^\circ_{n+1}-\tiltheta^\circ_{n}) 
			\\
								 & = \Delta^{\clA}_{n+2}\tiltheta^\circ_n 
						+ [ \clA_{n+1}\tiltheta^\circ_n - \clA_{n+2}\tiltheta^\circ_{n+1} ]
						+ \alpha_{n+1}\clA_{n+2}(A_{n+1}\tiltheta^\circ_{n} +\Delta_{n+1})
\end{aligned}
\]
The sequence $\{\clE_n\}$ from \eqref{e:clE-td-err} can be expressed as the sum
\[
\clE_n = \clE_n^{(1)} +  \clE_n^{(2)}+ \clE_n^{(3)}+ \clE_n^{(4)}
\]
where   $\clE_{n}^{(4)} = -\alpha_n\clA_{n+1}\tiltheta^\circ_{n}$,  and 
 the first three sequences are solutions to the following linear systems:
\begin{subequations}
\begin{align}
\clE_{n+1}^{(1)} & = \clE_{n}^{(1)}  + \alpha_{n+1}[A\clE_n^{(1)} + \Delta^{\clA}_{n+2}\tiltheta^\circ_{n}] \,,  && \clE_{0}^{(1)} = 0
\label{e:clE-td-err-decomp-1} \\
\clE_{n+1}^{(2)} & =\clE_{n}^{(2)} + \alpha_{n+1}[A\clE_n^{(2)} - \alpha_n [I+A]\clA_{n+1}\tiltheta^\circ_{n}] \,,  && \clE_1^{(2)} = \clA_1\tiltheta^\circ_0 
\label{e:clE-td-err-decomp-2} \\
\clE_{n+1}^{(3)} & =\clE_{n}^{(3)} + \alpha_{n+1}[A\clE_n^{(3)} + \alpha_{n+1} \clA_{n+2}(A_{n+1}\tiltheta^\circ_{n} +\Delta_{n+1})] \,,  && \clE_0^{(3)} =0 
\label{e:clE-td-err-decomp-3}
\end{align}
\label{e:clE-td-err-decomp}
\end{subequations}
The second recursion arises through the arguments used in the proof of \Lemma{t:Tele}.

Recall that $\lambda = -\varrho_0+ ui$ is an eigenvalue of the matrix $A$ with largest real part. For fixed $0<\varrho < \varrho_0$, let $T\geq 0$ denote the unique solution to the Lyapunov equation
\begin{equation}
[\varrho I + A]T + T[\varrho I + A]^\transpose + I = 0
\label{e:lyap-td}
\end{equation}
As previously, the norm of random vector $E\in \Re^d$ is defined as: $\|E\|_T = \sqrt{\Expect[E^\transpose TE]}$.

\begin{lemma}
\label{t:clE-td-err-inequalities}
Under Assumptions (A1)-(A4), there exist constants $L_{\ref{t:clE-td-err-inequalities}}$ and $K_{\ref{t:clE-td-err-inequalities}}$ such that, for all $n \ge 1$,
\begin{romannum}
\item The following holds for each $1\leq i \leq 3$,
\[
\|\clE_{n+1}^{(i)}\|_T^2 \leq (1-2\varrho\alpha_{n+1} + L_{\ref{t:clE-td-err-inequalities}}^2\alpha_{n+1}^2)\|\clE_{n}^{(i)}\|_T^2 + K_{\ref{t:clE-td-err-inequalities}}\alpha_{n+1}^2(\|\clE_n\|_T^2 + \|\tiltheta^\bullet_n\|_T^2 + 1)
\]
\item The following holds for $\clE_n^{(4)}$, 
\[
\|\clE_{n+1}^{(4)}\|_T^2  \leq K_{\ref{t:clE-td-err-inequalities}}\alpha_{n+1}^2( \|\clE_n\|_T^2 +\|\tiltheta^\bullet_n\|_T^2 +1)
\] 
\end{romannum}
\end{lemma}
The inequality below will be useful in  proving \Lemma{t:clE-td-err-inequalities}.
\begin{lemma}
	\label{t:cs-extended}
	For any real numbers $a, b$ and all $c>0$,
\[
	(a + b)^2 \leq (1+c^{-1})a^2+ (1+c)b^2
\]
\end{lemma}
\begin{proof}
	With $(a+b)^2 = a^2 + b^2 + 2ab$, the result follows directly from the inequality
\[
	2ab=  2 (a/\sqrt{c})(\sqrt{c}b) \leq a^2/c + cb^2                                                                                                   
\]
\end{proof}

\begin{proof}[Proof of \Lemma{t:clE-td-err-inequalities}]
First consider $\{\clE_n^{(1)}\}$ updated via \eqref{e:clE-td-err-decomp-1}. Since the martingale difference sequence $\Delta^{\clA}_{n+2}$ is uncorrelated with $\tiltheta^\circ_n$ or $\clE_n^{(1)}$, we have
\[
\|\clE_{n+1}^{(1)}\|_T^2 = \|[I+\alpha_{n+1}A]\clE_n^{(1)}\|_T^2 + \alpha_{n+1}^2 \|\Delta^{\clA}_{n+2}\tiltheta^\circ_{n}\|_T^2 
\]
Using the fact that $T\geq 0$ solves the Lyapunov equation \eqref{e:lyap-td} gives
\[
\|\clE_{n+1}^{(1)}\|_T^2 \leq 
(1-2\varrho\alpha_{n+1} + L_{1}^2\alpha_{n+1}^2)\|\clE_n^{(1)}\|_T^2 + \alpha_{n+1}^2 \|\Delta^{\clA}_{n+2}\tiltheta^\circ_{n}\|_T^2  
\]
where $L_{1}=\|A\|_T$ (the induced operator norm).  With $\tiltheta^\circ_{n} = \clE_n + \tiltheta^\bullet_n$,
\[
\|\Delta^{\clA}_{n+2}\tiltheta^\circ_{n}\|_T^2 \leq 2\|\Delta^{\clA}_{n+2}\|_T^2(\|\clE_n\|_T^2 + \|\tiltheta^\bullet_n\|_T^2)
\]
Consequently,
\begin{equation}
\|\clE_{n+1}^{(1)}\|_T^2 \leq 
(1-2\varrho\alpha_{n+1} + L_{1}^2\alpha_{n+1}^2)\|\clE_n^{(1)}\|_T^2 + K_1\alpha_{n+1}^2(\|\clE_n\|_T^2 + \|\tiltheta^\bullet_n\|_T^2)
\label{e:clE-td-err-1}
\end{equation}
where $K_1 = \sup_{n} 2 \|\Delta^{\clA}_{n+2}\|_T^2$ is finite by the $V$-uniform ergodicity of $\bfPhi$ applied to $\haclA_{i,j}^2$  (recall 
 \Theorem{t:Vuni}).

For $\{\clE_n^{(2)}\}$ updated by \eqref{e:clE-td-err-decomp-2}, using \Lemma{t:cs-extended} with $c=n(n+1)$ gives
\[
\begin{aligned}
\|\clE_{n+1}^{(2)}\|_T^2 
										 \leq & (1+\alpha_n\alpha_{n+1})(1-2\varrho\alpha_{n+1} + L_1^2\alpha_{n+1}^2)\|\clE_n^{(2)}\|_T^2 \\  
 										& + 2(\alpha_n\alpha_{n+1}+ \alpha_n^2\alpha_{n+1}^2)\|[I+A]\clA_{n+1}\|_T^2(\|\clE_n\|_T^2 + \|\tiltheta^\bullet_n\|_T^2)
\end{aligned}
\]
We can find $L_2$ and $K_2$ such that for all $n\geq 1$,
\[
\begin{aligned}
\alpha_{n+1}^2L_1^2 + \alpha_n\alpha_{n+1}(1-2\varrho\alpha_{n+1} + L_1^2\alpha_{n+1}^2) &\leq L_2^2\alpha_{n+1}^2 \\
                           2(\alpha_n\alpha_{n+1}+ \alpha_n^2\alpha_{n+1}^2)\|[I+A]\clA_{n+1}\|_T^2 &\leq  K_2 \alpha_{n+1}^2
\end{aligned}
\]
We then obtain the desired form for the sequence $\{\clE_n^{(2)}\}$
\begin{equation}
\begin{aligned}
\|\clE_{n+1}^{(2)}\|_T^2 \leq (1-2\varrho\alpha_{n+1} + L_2^2\alpha_{n+1}^2)\|\clE_n^{(2)}\|_T^2  + K_2\alpha_{n+1}^2(\|\clE_n\|_T^2 + \|\tiltheta^\bullet_n\|_T^2)
\end{aligned}
\label{e:clE-td-err-2}
\end{equation}

The same argument applies to $\{\clE_{n}^{(3)}\}$ in \eqref{e:clE-td-err-decomp-3}. Therefore,  for some constants $L_3$ and $K_3$,
\begin{equation}
\|\clE_{n+1}^{(3)}\|_T^2  \leq  (1-2\varrho\alpha_{n+1} + L_3^2\alpha_{n+1}^2)\|\clE_n^{(3)}\|_T^2   + K_3\alpha_{n+1}^2(\|\clE_{n}\|_T^2 + \|\tiltheta^\bullet_n\|_T^2 + 1)
\label{e:clE-td-err-3}
\end{equation}

A bound on the final term $\clE_{n+1}^{(4)}= -\alpha_{n+1}\clA_{n+2}\tiltheta^\circ_{n+1}$ is relatively easy.
\[
\begin{aligned}
\|\clE_{n+1}^{(4)}\|_T^2 
&= \|\alpha_{n+1}\clA_{n+2}[\tiltheta^\circ_{n}+ \alpha_{n+1}(A_{n+1}\tiltheta^\circ_{n}+\Delta_{n+1})]\|_T^2\\
& \leq  2 \alpha_{n+1}^2 \|\clA_{n+2}\|_T^2 ( \|I+\alpha_{n+1}A_{n+1}\|_T^2 \|\tiltheta^\circ_n\|_T^2  + \alpha_{n+1}^2 \|\Delta_{n+1}\|_T^2) 
\end{aligned}
\]
Hence there exists some constant $K_4$ such that
\[
\|\clE_{n+1}^{(4)}\|_T^2  \leq K_4 \alpha_{n+1}^2 (\|\clE_n\|_T^2 +\|\tiltheta^\bullet_n\|_T^2+1)
\]
\end{proof}

The results in \Lemma{t:clE-td-err-inequalities} lead to a rough bound on $\|\tiltheta^\circ_n\|_T^2$ presented in the following. This intermediate result will be used later to establish the refined bound in \Theorem{t:couple-td}.

\begin{lemma}
\label{t:clE-td-err}
Under Assumptions (A1)-(A4), 
\[
\limsup_{n\rightarrow\infty} n^{\varrho} \|\tiltheta^\circ_n\|_T^2  < \infty  \,, \qquad \text{for } \, \varrho < \varrho_0 \text{ and } \varrho \leq 1
\]
\end{lemma}

\begin{proof}
Denote $\clE^\tot_n = \sum_{i=1}^4\|\clE_{n}^{(i)}\|_T^2$. 
By \Lemma{t:clE-td-err-inequalities}, we can find $n_0 \ge 1$ such that $1-2\varrho\alpha_{n+1} + L_{\ref{t:clE-td-err-inequalities}}^2\alpha_{n+1}^2 > 0$ for $n\geq n_0$ and 
\[
\begin{aligned}
\clE^\tot_{n+1} 
			&\leq (1-2\varrho\alpha_{n+1} + L_{\ref{t:clE-td-err-inequalities}}^2\alpha_{n+1}^2)\clE^\tot_n + 4K_{\ref{t:clE-td-err-inequalities}}\alpha_{n+1}^2(\|\clE_n\|_T^2 + \|\tiltheta^\bullet_n\|_T^2 + 1) \\
			&\leq (1-2\varrho\alpha_{n+1} + L_{\ref{t:clE-td-err-inequalities}}^2\alpha_{n+1}^2)\clE^\tot_n + 4K_{\ref{t:clE-td-err-inequalities}}\alpha_{n+1}^2(4\clE^\tot_n + \|\tiltheta^\bullet_n\|_T^2 + 1)  \\
			&\leq (1-2\varrho\alpha_{n+1} + L_{\tot}^2\alpha_{n+1}^2)\clE^\tot_n + K_{\tot} \alpha_{n+1}^2
\end{aligned} 
\]
with $L_{\tot}^2=L_{\ref{t:clE-td-err-inequalities}}^2+16K_{\ref{t:clE-td-err-inequalities}}$ and $K_{\tot}  = \sup_n 4K_{\ref{t:clE-td-err-inequalities}}( \|\tiltheta^\bullet_n\|_T^2 + 1) $, which are finite by \Lemma{t:converge-sa} combined with \Lemma{t:clE-td-err-inequalities}.   
Iterating this inequality gives, for $n\geq n_0$,
\[
\clE^\tot_{n+1}
				\leq  \clE^\tot_{n_0}\prod_{k=n_0+1}^{n+1} (1-2\varrho\alpha_k+ L_{\tot}^2\alpha_k^2)  + K_{\tot} \sum_{k=n_0+1}^{n+1}  \alpha_k^{2}\prod_{l=k+1}^{n+1} (1-2\varrho \alpha_l + L_{\tot}^2\alpha_l^2) 
\]
By \Lemma{t:prod-n},
\[
\clE^\tot_{n+1}	\leq   \clE^\tot_{n_0} \frac{K_{\ref{t:prod-n}} n_0^{2\varrho}}{(n+2)^{2\varrho}}  + \frac{K_{\ref{t:prod-n}} K_{\tot}}{(n+2)^{2\varrho}} \sum_{k=n_0+1}^{n+1} k^{2\varrho-2}
\]
The partial sum can be estimated by an integral: with $2\varrho - 2 \leq 0$,
\begin{equation}
\sum_{k=n_0}^{n+1} k^{2\varrho - 2} \leq 1+ \int_{n_0}^{n+1} r^{2\varrho -2}dr = 
\begin{cases}
1 + [(n+1)^{2\varrho-1} - n_0^{2\varrho-1}]/(2\varrho - 1)  \,,    & \text{if } \varrho \neq \half \\
1 + \ln(n+1) - \ln(n_0)                \,,   & \text{if } \varrho = \half
\end{cases}
\label{e:partsum-int}
\end{equation}
Given $\varrho \leq 1$,
\[
n^\varrho \clE^\tot_n \leq \clE^\tot_{n_0} \frac{K_{\ref{t:prod-n}} n_0^{2\varrho}}{(n+2)^{\varrho}}  + \frac{K_{\ref{t:prod-n}} K_{\tot}}{(n+2)^{\varrho}} \sum_{k=n_0+1}^{n+1} k^{2\varrho-2} < \infty
\]
Consequently, $\limsup_{n\rightarrow\infty}n^{\varrho}\|\clE_n\|_T^2 < \infty$ by the inequality $n^{\varrho}\|\clE_n\|_T^2\leq 4n^{\varrho}\clE^\tot_n$. Then we have
\[
n^{\varrho}\|\tiltheta^\circ_n\|_T^2 \leq 2n^{\varrho}\|\clE_n \|_T^2 + 2n^{\varrho}\| \tiltheta^\bullet_n \|_T^2 
\]
where $n^{\varrho}\| \tiltheta^\bullet_n \|_T^2 \rightarrow 0$ as $n$ goes to infinity by \Lemma{t:converge-sa}. Hence $\limsup_{n\rightarrow\infty}n^{\varrho}\|\tiltheta^\circ_n\|_T^2 < \infty$.
\end{proof}

\begin{proof}[Proof of \Theorem{t:couple-td}]
First consider $\{\clE_n^{(2)}\}$ updated via \eqref{e:clE-td-err-decomp-2}. 
By the triangle inequality and the inequality $\sqrt{1-x}\leq \half x$,  
\[
\begin{aligned}
\|\clE_{n+1}^{(2)}\|_T   
									 &	\leq \|[I+\alpha_{n+1}A]\clE_n^{(2)}\|_T  + \alpha_n\alpha_{n+1} \|[I+A]\clA_{n+1}\tiltheta^\circ_{n}\|_T \\
									 & \leq  (1-\varrho\alpha_{n+1} + \half L^2\alpha_{n+1}^2) \|\clE_n^{(2)}\|_T  + \alpha_{n+1}^{2+\varrho/2} K
\end{aligned}
\]
where $L= \|A\|_T$ and $K=\sup_n 2\|[I+A]\clA_{n+1}\|_T \|\tiltheta^\circ_n\|/(n+1)^{\varrho/2}$, which is finite thanks to \Lemma{t:clE-td-err}. Hence, by \Lemma{t:prod-n} once more, 
\[
\begin{aligned}
\|\clE_{n+1}^{(2)}\|_T &  \leq \|\clE_1^{(2)}\|_T \prod_{k=2}^{n+1}[1-\varrho\alpha_{k} + \half L^2\alpha_{k}^2]  + K\sum_{k=2}^{n+1} \alpha_{k}^{2+\varrho/2}\prod_{l=k+1}^{n+1}[1-\varrho\alpha_{k} + \half L^2\alpha_{k}^2] \\	
									 &	\leq \|\clE_1^{(2)}\|_T \frac{K_{\ref{t:prod-n}} }{(n+2)^\varrho} +  \frac{KK_{\ref{t:prod-n}} }{(n+2)^\varrho}\sum_{k=2}^{n+1} k^{\varrho/2-2}
\end{aligned}
\]
With $\varrho \leq 1$, we have $\sum_{k=1}^{\infty} k^{\varrho/2-2} \leq \sum_{k=1}^{\infty} k^{-3/2}  < \infty$. Hence $
\limsup_{n\rightarrow \infty}n^{\varrho} \|\clE_n^{(2)}\|_T< \infty$. Replacing $A_{n+1}\tiltheta^\circ_{n} +\Delta_{n+1}$ with $\tiltheta^\circ_{n+1} - \tiltheta^\circ_{n}$ in \eqref{e:clE-td-err-decomp-3}, the same argument applies to $\{\clE_n^{(3)}\}$ and we get $\limsup_{n\rightarrow \infty}n^\varrho \|\clE_n^{(3)}\|_T< \infty$. 
The fact that $\limsup_{n\rightarrow \infty} n\|\clE_{n+1}^{(4)}\|_T < \infty$ follows directly from definition $\clE_{n}^{(4)} = -\alpha_n\clA_{n+1}\tiltheta^\circ_{n}$ and \Lemma{t:clE-td-err}. Then we have, for each $2\leq i \leq4$,
\begin{equation}
\label{e:clE-td-err-rate-23}
\limsup_{n\rightarrow \infty} n^\varrho\|\clE_n^{(i)}\|_T < \infty\,, \qquad   \text{for } \varrho <\varrho_0 \text{ and } \varrho \leq 1
\end{equation}

Now consider the martingale difference part $\{\clE_n^{(1)}\}$.
The following is directly obtained from \eqref{e:clE-td-err-decomp-1}:  
%\spm{This might have gone in a lemma?   The arguments are very repetitive.}
\[
\begin{aligned}
\|\clE_{n+1}^{(1)}\|_T^2 
									\leq & (1-2\varrho\alpha_{n+1} + L^2\alpha_{n+1}^2)\|\clE_n^{(1)}\|_T^2 + \alpha_{n+1}^2\|\Delta^{\clA}_{n+2}\|_T^2 \|\tiltheta^\circ_{n}\|_T^2 \\
									\leq & (1-2\varrho\alpha_{n+1} + L^2\alpha_{n+1}^2)\|\clE_n^{(1)}\|_T^2 + \alpha_{n+1}^2 \|\Delta^{\clA}_{n+2}\|_T^2 \bigl[8\sum_{i=1}^4\|\clE_n^{(i)}\|_T^2 + 2\|\tiltheta^\bullet_n\|_T^2\bigr]\\
\end{aligned}
\]
From \Lemma{t:converge-sa} we have $\sup_n n^{\delta}\|\tiltheta^\bullet_n\|_T^2 < \infty$ for $\delta = \min(1, 2\varrho)$. Combining this with \eqref{e:clE-td-err-rate-23} implies that there exists some constant $K_{\clM}$ such that for $\delta = \min(1, 2\varrho)$,
\[
\|\Delta^{\clA}_{n+2}\|_T^2 \bigl[8\sum_{i=2}^4\|\clE_n^{(i)}\|_T^2 + 2\|\tiltheta^\bullet_n\|_T^2\bigr] \leq K_{\clM}\frac{1}{(n+1)^{\delta}} 
\]
Consequently,
\[
\begin{aligned}
\|\clE_{n+1}^{(1)}\|_T^2 
									& \leq  (1-2\varrho\alpha_{n+1} + L_{\clM}^2\alpha_{n+1}^2)\|\clE_n^{(1)}\|_T^2 +K_{\clM}\alpha_{n+1}^{2+\delta}
\end{aligned}
\]
where $ L_{\clM}^2 = \sup_n L^2 +8\|\Delta^{\clA}_{n+2}\|_T^2$. With initial condition $\clE_0=0$, iterating this inequality gives
\[
\begin{aligned}
\|\clE_{n+1}^{(1)}\|_T^2  
		\leq   K_{\clM}\sum_{k=1}^{n+1} \alpha_k^{2+\delta}\prod_{l=k+1}^{n+1}[1-2\varrho\alpha_{l} +  L_{\clM}^2\alpha_{l}^2]\leq \frac{K_{\clM}K_{\ref{t:prod-n}} }{(n+2)^{2\varrho}}\sum_{k=1}^{n+1} k^{-(2+\delta- 2\varrho)}
\end{aligned}
\]
With $2+\delta -2\varrho > 0$, the partial sum is bounded by an integral similar as \eqref{e:partsum-int}:
\[
\frac{1}{(n+2)^{2\varrho}}\sum_{k=1}^{n+1} k^{-(2+\delta- 2\varrho)} =
\begin{cases}
O((n+1)^{-2\varrho}), & \text{if } \varrho \leq \half \text{ and } \delta= 2\varrho \\
O((n+1)^{-2\varrho}), & \text{if } \half <\varrho < 1 \text{ and } \delta= 1 \\
O((n+1)^{-2}), & \text{if } \varrho > 1 \text{ and } \delta= 1
\end{cases}
\]
Therefore, 
\begin{romannum}
\item If $\varrho_0 \leq 1$, then $\limsup_{n\rightarrow\infty}(n+1)^{2\varrho}\|\clE_{n+1}^{(1)}\|_T^2 <\infty$ for $\varrho < \varrho_0$.
\item If $\varrho_0 > 1$, then $\limsup_{n\rightarrow\infty}(n+1)^{2}\|\clE_{n+1}^{(1)}\|_T^2 <\infty$.
\end{romannum}
Given that the same convergence rates hold for the other components in \eqref{e:clE-td-err-rate-23}, the conclusion then follows.
\end{proof}

\end{document}